\documentclass[11pt,a4paper]{elsarticle}

\usepackage[utf8]{inputenc}

\usepackage{amsmath}
\usepackage{amsthm}
\usepackage{enumerate}
\usepackage{mathdots}
\usepackage{amsfonts}
\usepackage{multicol}
\usepackage{amsmath}
\usepackage{amssymb}
\usepackage{fancybox}
\usepackage[usenames,dvipsnames]{color}
\usepackage{graphicx}
\usepackage{tikz}
\usepackage[left=3cm,top=4cm,right=3cm,bottom=3.5cm]{geometry}
\usepackage{pifont}
\usepackage{natbib}
\usepackage{geometry}
\usepackage{arydshln}

\makeatletter
\def\bbordermatrix#1{\begingroup \m@th
	\global\let\perhaps@scriptstyle\scriptstyle
	\@tempdima 4.75\p@
	\setbox\z@\vbox{%
		\def\cr{%
			\crcr
			\noalign{%
				\kern2\p@
				\global\let\cr\endline
				\global\let\perhaps@scriptstyle\relax
			}%
		}%
		\ialign{$\make@scriptstyle{##}$\hfil\kern2\p@\kern\@tempdima
			&\thinspace\hfil$\perhaps@scriptstyle##$\hfil
			&&\quad\hfil$\perhaps@scriptstyle##$\hfil\crcr
			\omit\strut\hfil\crcr
			\noalign{\kern-\baselineskip}%
			#1\crcr\omit\strut\cr}}%
	\setbox\tw@\vbox{\unvcopy\z@\global\setbox\@ne\lastbox}%
	\setbox\tw@\hbox{\unhbox\@ne\unskip\global\setbox\@ne\lastbox}%
	\setbox\tw@\hbox{$\kern\wd\@ne\kern-\@tempdima\left[\kern-\wd\@ne
		\global\setbox\@ne\vbox{\box\@ne\kern2\p@}%
		\vcenter{\kern-\ht\@ne\unvbox\z@\kern-\baselineskip}\,\right]$}%
	\null\;\vbox{\kern\ht\@ne\box\tw@}\endgroup}
\def\make@scriptstyle#1{\vcenter{\hbox{$\scriptstyle#1$}}}
\makeatother

\newcommand{\R}{\mathbb{R}}

\newcommand{\C}{\mathbb{C}}

\newcommand{\efe}{\mathbb{F}}
\newcommand{\FF}{\mathbb{F}}
\newcommand{\F}{\mathbb{F}}

\newcommand{\la}{\lambda}

\newcommand{\M}{\mathbb{M}}
\newcommand{\DM}{\mathbb{DM}}
\newcommand{\bS}{\mathbb{S}}

\def\rank{\mathop{\rm rank}\nolimits}
\newcommand{\wh}{\widehat}

\newtheorem{theo}{Theorem}[section]

\newtheorem{deff}[theo]{Definition}

\newtheorem{prop}[theo]{Proposition}

\newtheorem{lem}[theo]{Lemma}

\newtheorem{con}[theo]{Corollary}

\newtheorem{rem}[theo]{Remark}

\newtheorem{example}[theo]{Example}

\DeclareMathOperator{\diag}{diag}
\DeclareMathOperator{\rev}{rev}

\begin{document}
\title{Strong linearizations of rational matrices with polynomial part expressed in an orthogonal basis}

\begin{abstract}
	We construct a new family of strong linearizations of rational matrices considering the polynomial part of them expressed in a basis that satisfies a three term recurrence relation. For this purpose, we combine the theory developed by Amparan et al., MIMS EPrint 2016.51, and the new linearizations of polynomial matrices introduced by Fa{\ss}bender and Saltenberger, Linear Algebra Appl., 525 (2017). In addition, we present a detailed study of how to recover eigenvectors of a rational matrix from those of its linearizations in this family. We complete the paper by discussing how to extend the results when the polynomial part is expressed in other bases, and by presenting strong linearizations that preserve the structure of symmetric or Hermitian rational matrices. A conclusion of this work is that the combination of the results in this paper with those in Amparan et al.,  MIMS EPrint 2016.51, allows us to use essentially all the strong linearizations of polynomial matrices developed in the last fifteen years to construct strong linearizations of any rational matrix by expressing such matrix in terms of its polynomial and strictly proper parts.   
\end{abstract}

\begin{keyword}
	rational matrix \sep rational eigenvalue problem \sep strong block minimal bases pencil \sep strong linearization \sep recovery of eigenvectors, symmetric strong linearization, Hermitian strong linearization
	 
\medskip\textit{AMS subject classifications}: 65F15, 15A18, 15A22, 15A54, 93B18, 93B20, 93B60
\end{keyword}

\date{ }
\author[uc3m]{Froil\'an M. Dopico\fnref{fn1}}
\ead{dopico@math.uc3m.es}
\author[upv]{Silvia Marcaida\corref{cor1}\fnref{fn2}}
\ead{silvia.marcaida@ehu.eus}
\author[uc3m]{María C. Quintana\fnref{fn1}}
\ead{maquinta@math.uc3m.es}

\cortext[cor1]{Corresponding author}

\address[uc3m]{Departamento de Matem\'aticas,
	Universidad Carlos III de Madrid, Avda. Universidad 30, 28911 Legan\'es, Spain.}
\address[upv]{Departamento de Matem\'{a}tica Aplicada y Estadística e Investigación Operativa,
	Universidad del Pa\'{\i}s Vasco UPV/EHU, Apdo. Correos 644, Bilbao 48080, Spain.}

\fntext[fn1]{Supported by ``Ministerio de Econom\'ia, Industria y Competitividad (MINECO)'' of Spain and ``Fondo Europeo de Desarrollo Regional (FEDER)'' of EU through grants MTM2015-65798-P and MTM2017-90682-REDT. The research of M. C. Quintana is funded by the “contrato predoctoral” BES-2016-076744 of MINECO.} 
\fntext[fn2]{Supported by ``Ministerio de Econom\'ia, Industria y Competitividad (MINECO)'' of Spain and ``Fondo Europeo de Desarrollo Regional (FEDER)'' of EU through grants MTM2017-83624-P and MTM2017-90682-REDT, and by UPV/EHU through grant GIU16/42.}

\maketitle

\section{Introduction}

In recent years, the interest in solving the rational eigenvalue problem (REP) has grown as it arises in many applications, either directly or as an approximation of other nonlinear eigenvalue problems, see \cite{eigenmode,guttel-tisseur,guttel-NLEIGS,automatic,Mehrmannnonlineareigenvalue,SuBai}. There are several algorithms for its numerical resolution and, since the appearance of \cite{SuBai}, using linearizations is one of the most competitive methods for solving REPs nowadays \cite{ratkrylov,cork}. This has led to a rigorous development of the theory of linearizations for rational matrices.

There are two different approaches in order to give a notion of linearization of a rational matrix. On the one hand, Alam and Behera give in \cite{AlBe16} a definition based on the fact that any rational matrix $G(\la)$ admits a right coprime matrix fraction description $G(\la)=N(\la)D(\la)^{-1},$ where $N(\la)$ and $D(\la)$ are polynomial matrices.  These linearizations preserve the finite pole and zero structure of the original matrix. In contrast, Amparan et al. give in \cite{strong} a new notion of linearization that not only preserves the finite but also the infinite structure of poles and zeros. These linearizations are called strong linearizations in \cite{strong}. This definition and other notions about rational matrices will be reviewed in Section 2.

Throughout this work, it is fundamental the fact that any rational matrix $G(\la)$ can be uniquely written as $G(\lambda)=D(\lambda)+G_{sp}(\lambda)$ where $D(\lambda)$ is a polynomial matrix, called the polynomial part of $G(\la),$ and $G_{sp}(\lambda)$ is a strictly proper rational matrix, called the strictly proper part of $G(\la).$ Thanks to this property, infinitely many strong linearizations of rational matrices are constructed in \cite{strong} considering ``strong block minimal bases pencils'' associated to their polynomial parts, see \cite{BKL}. These strong block minimal bases pencils are strong linearizations of the polynomial part and according to \cite{fiedler} include, modulo permutations, all the Fiedler-like linearizations of the polynomial part. Although strong block minimal bases pencils are one of the most important class of strong linearizations of polynomial matrices, the question whether or not other strong linearizations of rational matrices can be constructed  based on another kinds of strong linearizations of the polynomial part arises naturally. 

For answering the question posed in the previous paragraph, we construct in this paper strong linearizations of a rational matrix by using strong linearizations of its polynomial part $D(\la)$ that belong to the other important family of strong linearizations of polynomial matrices (which are not strong block minimal bases pencils in general), i.e., the so-called vector spaces of linearizations, originally introduced in \cite{MMMM}, further studied in \cite{singular,bivariate}, and recently extended in \cite{ortho}. In particular, we consider in this paper strong linearizations of $D(\lambda)$ that belong to the ansatz spaces $\mathbb{M}_{1}(D)$ or $\mathbb{M}_{2}(D),$ developed by Fa{\ss}bender and Saltenberger in \cite{ortho}. Therefore, the results in this paper are of interest when rational matrices with nontrivial polynomial part are considered, that is, rational matrices in which the polynomial part has degree greater than or equal to two.

As a consequence of the discussion above, we emphasize the following main conclusion of this work: the combination of the results in this paper and those in \cite{strong} allows us to construct very easily infinitely many strong linearizations of rational matrices via the following three-step strategy: (1) express the rational matrix as the sum of its polynomial and strictly proper parts; (2) construct \textit{any} of the strong linearizations of the polynomial part known so far; and (3) combine adequately that strong linearization with a minimal state-space realization of the stricly proper part. 

Next, another motivation of the results in this paper is discussed. In order to compute the eigenvalues of polynomial matrices from linearizations, the work \cite{Kressner-roman} shows that, for polynomial matrices of large degree, the use of the monomial basis to express the matrix leads to numerical instabilities. According to the algorithms in \cite{ratkrylov,SuBai,cork}, it is expected that this instability appears also while computing eigenvalues of REPs when the polynomial part of the rational matrix has large degree and is expressed in terms of the monomial basis. For that reason, it is of interest to consider rational matrices with polynomial parts expressed in other bases as the Chebyshev basis. In particular, in Sections \ref{sect:m1} and \ref{sect:m2}, we construct strong linearizations of rational matrices with polynomial parts expressed in terms of a basis that satisfies a three term recurrence relation. In addition, in Section \ref{sect:other}, we briefly discuss how to construct strong linearizations when the polynomial part is expressed in other bases. 
We emphasize that the construction of these new strong linearizations is a consequence of the theory of strong linearizations developed in \cite{strong} together with Lemma \ref{more}. More precisely, given a strong linearization of a rational matrix, Lemma \ref{more} allows us to obtain infinitely many strong linearizations of the rational matrix by using strict equivalence with a certain structure.  

The rest of this paper is organized as follows. In Section \ref{recovery}, we show how to recover the eigenvectors of the rational matrix from those of its strong linearizations constructed in Sections \ref{sect:m1} and \ref{sect:m2}. Moreover, given a symmetric rational matrix, in Section \ref{sect:sym} we construct strong linearizations that preserve its symmetric structure by using symmetric realizations of the strictly proper part, which are introduced in Section \ref{sect:realsym}, and strong linearizations in the double ansatz space $\DM(D)$ \cite{ortho} of the polynomial part. In Section \ref{sect:herm}, we present analogous results for Hermitian rational matrices. Finally, Section \ref{sect:con} is reserved for discussing the conclusions and lines of future work.

\section{Preliminaries}\label{prelim}

$\efe[\la]$ denotes the ring of polynomials with
coefficients in an arbitrary field $\efe,$ and $\efe(\la)$ the field of \textit{rational functions}, i.e., the field of fractions of $\efe[\la]$. $\efe(\la)^{p\times m},$ $\efe[\la]^{p\times m}$ and $ \efe^{p\times m}$ denote the sets of $p\times m$ matrices with elements in $\efe(\la),$ $\efe[\la]$ and $\efe,$ respectively. The elements of $\efe(\la)^{p\times m}$ and $\efe[\la]^{p\times m}$ are called rational and polynomial matrices, respectively. A polynomial matrix $P(\lambda)=\sum_{i=0}^{k}\lambda^{i}P_{i}$ with $P_{i}\in\efe^{p\times m}$ is said to have \textit{degree} $k$ if $P_{k}\neq 0.$ If $k=1$ or $k=0$ then $P(\la)$ is said to be a \textit{pencil}. Matrices in $\efe[\la]^{m\times m}$ with nonzero constant determinant are said to be \textit{unimodular}. Two rational matrices $Q(\lambda),R(\lambda)\in \efe(\la)^{p\times m}$ are said to be \textit{unimodularly equivalent} if there exist unimodular matrices $U(\lambda)\in \efe[\la]^{p\times p}$ and $V(\lambda)\in\efe[\la]^{m\times m}$ such that $U(\la)Q(\la)V(\la)=R(\la).$ Moreover, $Q(\lambda)$ and $R(\lambda)$ are said to be \textit{strictly equivalent} if $UQ(\la)V=R(\la)$ with $U\in \efe^{p\times p}$ and $V\in \efe^{m\times m}$ invertible matrices.

A rational matrix $G(\la)\in\F(\la)^{p\times m}$ is said to be \textit{regular} or \textit{nonsingular} if $p=m$ and its determinant, $\det G(\la),$ is not identically equal to zero. Otherwise, $G(\la)$ is said to be \textit{singular}. Given any rational matrix $G(\la)\in\F(\la)^{p\times m},$ \textit{the (finite) eigenvalues of $G(\la)$} are defined as the scalars $\la_0\in\overline{\F}$ (the algebraic closure of $\FF$) such that $G(\la_0)\in\overline{\F}^{p\times m}$ and $\rank G(\la_0)< \displaystyle\max_{\mu\in\overline{\F}}\rank G(\mu).$ The \textit{rational eigenvalue problem} (REP) consists of finding the eigenvalues of $G(\la).$ If $G(\la)\in\F(\la)^{m\times m}$ is regular, which is the most common case in applications of REPs, the REP is equivalent to the problem of finding scalars $\la_0\in\overline{\F}$ such that there exist nonzero constant vectors $x\in\overline{\efe}^{m\times 1}$ and $y\in\overline{\efe}^{m\times 1}$ satisfying $$G(\la_0)x=0 \quad\mbox{and}\quad y^{T}G(\la_0)=0,$$
respectively.
The vectors $x$ are called \textit{right eigenvectors associated to $\la_0$}, and the vectors $y$ \textit{left eigenvectors}. Although it is not common in the literature, if $G(\la)\in\F(\la)^{p\times m}$ is singular, we call in this paper right and left eigenvectors of $G(\la)$ associated to an eigenvalue $\la_0$ to any nonzero vectors $x\in\overline{\efe}^{m\times 1}$ and $y\in\overline{\efe}^{p\times 1}$ satisfying  $G(\la_0)x=0$ and  $y^{T}G(\la_0)=0,$ respectively.

 The finite poles and zeros of a rational matrix $G(\la)$ are the roots in $\overline{\FF}$ of the polynomials that appear on the denominators and numerators, respectively, in its \textit{(finite) Smith--McMillan form} (see \cite{strong, Rosen70, Vard91}). Then the finite eigenvalues of $G(\la)$ are the finite zeros that are not poles.

 For solving the REP, and many other problems on rational matrices, it is useful to consider the fact that any rational matrix $G(\la)\in\FF(\la)^{p\times m}$  can be written as
\begin{equation}\label{s1.eqrealizG}
G(\la)=D(\la)+C(\la)A( \la)^{-1}B(\la)
\end{equation}
for some nonsingular polynomial matrix  $A(\la)\in\FF[\la]^ {n\times n}$ and polynomial matrices
$B(\la)\in\FF[\la]^{n\times m}$, $C(\la)\in\FF[\la]^{p\times n}$ and $D(\la)\in\FF[\la]^{p\times m}$ (see \cite{Rosen70}). The polynomial matrix
\begin{equation}\label{s1.eqpolsysmat}
P(\la)=\begin{bmatrix}
A(\la) & B(\la)\\
-C(\la) & D(\la)
\end{bmatrix}
\end{equation}
is called a \textit{polynomial system matrix} of $G(\la)$, i.e, $G(\la)$ is the Schur complement of $A(\la)$ in $P(\la)$. Then
$G(\la)$ is called the \textit{transfer function matrix} of
$P(\la),$ and  $\deg(\det A(\la))$ the \textit{order} of $P(\la),$  where $\deg(\cdot)$ stands for degree. Moreover, $P(\la)$ is said to have \textit{least order}, or to be \textit{minimal},
if its order is the smallest integer for which polynomial matrices $A(\la)$ (nonsingular), $B(\la)$,
$C(\la)$ and $D(\la)$ satisfying (\ref{s1.eqrealizG}) exist. The least order is uniquely determined
by $G(\la)$ and is denoted by $\nu(G(\la))$. It is also called the \textit{least order} of $G(\la)$
(\cite[Chapter 3, Section 5.1]{Rosen70} or \cite[Section 1.10]{Vard91}). From \cite[Chapter 3, Theorem 4.1]{Rosen70}, it can be deduced that $\nu(G(\la))$ is the degree of the polynomial that results by making the product of the denominators in the (finite) Smith--Mcmillan form of $G(\la).$  The regular REP $G(\la)x=0$ is related to the polynomial eigenvalue problem (PEP) $P(\la)z=0$ as is shown in \cite[Proposition 3.1]{strong}.

A rational function $r(\la)=\frac{n(\la)}{d(\la)}$ is said to be \textit{proper} if $\deg(n(\la))\leq\deg(d(\la)),$ and \textit{strictly proper} if $\deg(n(\la))<\deg(d(\la)).$ Let us denote $\F_{pr}(\lambda)$ the ring of proper rational functions. Its units are called \textit{biproper rational functions}, i.e., rational functions having the same degree of numerator and denominator. $\F_{pr}(\lambda)^{p\times m}$ denotes the set of $p\times m$ matrices with entries in $\F_{pr}(\lambda),$ which are called \textit{proper matrices}. A \textit{biproper matrix} is a square proper matrix whose determinant is a biproper rational function.

By the division algorithm for polynomials, any rational function $r(\la)\in\FF(\la)$ can be uniquely written as $r(\la)=p(\la)+r_{sp}(\la),$
where $p(\la)$ is a polynomial and $r_{sp}(\la)$ a strictly proper rational function. Therefore, any rational matrix $G(\la)\in\FF(\la)^{p\times m}$ can be uniquely written as
\begin{equation}\label{eq.polspdec}
G(\la)=D(\la)+G_{sp}(\la)
\end{equation}
where $D(\la)\in\FF[\la]^{p\times m}$ is a polynomial matrix and $G_{sp}(\la)\in \F_{pr}(\lambda)^{p\times m}$ is a
\textit{strictly proper rational matrix}, i.e., the entries of $G_{sp}(\la)$ are strictly
proper rational functions. As said in the introduction, $D(\la)$ is called the polynomial part of $G(\la)$ and $G_{sp}(\la)$ its strictly proper part.

The polynomial system matrix $P(\la)$ of $G(\la)$ is said to be a polynomial system matrix in
\textit{state-space form} if $A(\la)=\la I_n -A$, $B(\la)=B$ and $C(\la)=C$ for some constant matrices $A\in\FF^{n\times n}$, $B\in\FF^{n\times m}$ and
$C\in\FF^{p\times n}.$ It is known that any strictly proper rational matrix admits \textit{state-space realizations} (see \cite{Rosen70}
or \cite{Kailath80}). This means that for some positive integer $n$ there exist constant matrices $A\in\FF^{n\times n}$, $B\in\FF^{n\times m}$ and
$C\in\FF^{p\times n}$ such that $G_{sp}(\la)= C(\la I_n-A)^{-1}B$ and
$$
\begin{bmatrix}
\la I_n-A & B\\
-C & D(\la)
\end{bmatrix}
$$
is a polynomial system matrix of $G(\la)$. Therefore $G(\la)=D(\la)+C(\la I_n-A)^{-1}B.$
In addition, the state-space realization may always be taken of least order, or minimal, (i.e., such that the polynomial
system matrix in state-space form is of least order).

Rational matrices may also have infinite eigenvalues. In order to define them, we need the notion of reversal. 

\begin{deff}\label{reversal}
	Let $G(\lambda)\in\F(\lambda)^{p\times m}$ be a rational matrix expressed in the form \eqref{eq.polspdec}. We define the reversal of $G(\la)$ as the rational matrix
	\begin{equation*} 
	\rev G(\lambda)=\lambda^{d}G\left(\dfrac{1}{\lambda}\right)
	\end{equation*}
	where $d= \deg (D(\lambda))$ if $G(\lambda)$ is not strictly proper, and $d=0$ otherwise. 
\end{deff}

Notice that this definition extends the definition of reversal for polynomial matrices (see \cite[Definition 2.12]{spectral} or \cite[Definition 2.2]{MMMM}). Moreover, note that $d=0$ if and only if $G(\lambda)$ is proper. Following the usual definition in polynomial matrices \cite[Definition 2.3]{MMMM}, we say that $G(\la)$ has an \textit{eigenvalue at infinity} if $\rev G(\lambda)$ has an eigenvalue at $\la=0.$ If $G(\la)$ has an eigenvalue at infinity, we say that $z$ is a \textit{right} (respectively \textit{left}) \textit{eigenvector associated to infinity} if $z$ is a right (respectively left) eigenvector associated to $0$ of $\rev G(\lambda).$

\begin{rem} \rm \textit{Poles} and \textit{zeros at infinity} of a rational matrix $G(\la)$ are defined as the poles and zeros at $\la=0$ of $G(1/\la)$ (see \cite{Kailath80}). If $G(\lambda)$ is not proper, i.e., $\deg (D(\lambda))\geq 1,$ $G(\lambda)$ has always a pole at $\infty$ (see \cite{AmMaZa15}). Thus, if we define the eigenvalues at infinity of $G(\la)$ as those zeros that are not poles at infinity, any non-proper $G(\la)$ would not have eigenvalues at infinity. In particular, this would happen if $G(\la)$ is a polynomial matrix. Therefore, as in the polynomial case, we have considered $\rev G(\lambda)$ in order to define eigenvalues at infinity.
\end{rem}

Next we present the definition of strong linearization for a rational matrix given in \cite{strong}. This definition contains the notion of first invariant order at infinity $q_{1}$ of a rational matrix $G(\lambda).$ For any non strictly proper rational matrix  this number is $-\deg(D(\lambda))$ where $D(\lambda)$ is the polynomial part of $G(\lambda)$ in the expression \eqref{eq.polspdec}; otherwise, $q_{1} > 0.$ More information can be found in \cite{AmMaZa15,strong,Vard91}. 

\begin{deff}\textbf{\cite[Definition 6.2]{strong}}\label{def_stronglin}
	Let $G(\la) \in\F(\la)^{p\times m}$. Let $q_1$ be its first invariant order at infinity and $g=\min(0,q_1)$.
	Let $n=\nu(G(\la))$. A strong linearization of $G(\la)$ is a linear polynomial matrix
	\begin{equation}\label{eq_lin_inf}
		\mathcal{L}(\la)=\left[\begin{array}{cc}
			A_1 \la +A_0 &B_1 \la +B_0\\-(C_1 \la +C_0)&D_1 \la +D_0
		\end{array}\right]\in\F[\la]^{(n+q)\times (n+r)}
	\end{equation}
	such that the following conditions hold:
	\begin{itemize}
		\item[(a)] if $n>0$ then $\det(A_1\la+A_0)\neq 0$, and
		\item[(b)] if $\wh{G}(\la)=(D_1\la+D_0)+(C_1\la+C_0)(
		A_1\la+A_0)^{-1}(B_1\la+B_0)$, $\wh{q}_{1}$ is its first invariant order at infinity and $\wh{g}=\min(0,\wh{q}_1)$ then:
		\begin{itemize}
			\item [(i)] there exist nonnegative integers $s_1,s_2,$ with $s_1-s_2=q-p=r-m,$ and unimodular matrices $U_1(\la)\in\F[\la]^{(p+s_1)\times (p+s_1)}$ and
			$U_2(\la)\in\F[\la]^{(m+s_1)\times (m+s_1)}$ such that $$U_1(\la)\diag(G(\la),I_{s_1})U_2(\la)=\diag(\wh{G}(\la),I_{s_2})\text{, and}$$
			\item [(ii)] there exist biproper matrices $B_1(\la)\in\F_{pr}(\la)^{(p+s_1)\times (p+s_1)}$ and $B_2(\la)\in$\\ $\F_{pr}(\la)^{(m+s_1)\times (m+s_1)}$ such that  
			$$B_1(\la)\diag(\la^{g}G(\la),I_{s_1})B_2(\la)=\diag(\la^{\wh{g}}\wh{G}(\la),I_{s_2}).$$
		\end{itemize}
	\end{itemize}
\end{deff}

It may seem that the integer $\nu(G(\la))$ has to be previously known in order to verify that a linear polynomial matrix as in \eqref{eq_lin_inf} is a strong linearization of $G(\la)$. However, there are conditions to ensure that the size of $A_1\la+A_0$ is $n=\nu(G(\la)).$ We state them in Proposition \ref{leastorder}.
\begin{prop}\label{leastorder}
 Let
 \begin{equation*}
 \mathcal{L}(\la)=\left[\begin{array}{cc}
 A_1 \la +A_0 &B_1 \la +B_0\\-(C_1 \la +C_0)&D_1 \la +D_0
 \end{array}\right]\in\F[\la]^{(n+q)\times (n+r)}
 \end{equation*} be a linear polynomial matrix with $n>0$ and $\det(A_1\la+A_0)\neq 0.$ Assume that there exist nonnegative integers $s_1,s_2,$ with $s_1-s_2=q-p=r-m,$ and unimodular matrices $U_1(\la)\in\F[\la]^{(p+s_1)\times (p+s_1)}$ and
 $U_2(\la)\in\F[\la]^{(m+s_1)\times (m+s_1)}$ such that
 \begin{equation}\label{eq_unimodular}
U_1(\la)\diag(G(\la),I_{s_1})U_2(\la)=\diag(\wh{G}(\la),I_{s_2}),
 \end{equation}
 where $\wh{G}(\la)=(D_1\la+D_0)+(C_1\la+C_0)(
 A_1\la+A_0)^{-1}(B_1\la+B_0).$ Then $n=\nu(G(\la))$ if and only if the following conditions hold:
 \begin{itemize}
 	\item[a)] $A_{1}$ is invertible, and
 	\item[b)] $\rank\begin{bmatrix} A_1 \mu + A_0 \\ C_1 \mu + C_0 \end{bmatrix}=\rank\begin{bmatrix} A_1 \mu + A_0  & B_1 \mu + B_0  \end{bmatrix}=n$ for all $\mu\in\overline{\F}.$ 
 \end{itemize}  
\end{prop}
\begin{proof} Condition $b)$ is equivalent to $\mathcal{L}(\la)$ being a minimal polynomial system matrix, since $\det(A_1\la+A_0)\neq 0,$ see \cite[Chapters 2 and 3]{Rosen70}. By condition \eqref{eq_unimodular} and \cite[Lemma 3.4]{strong}, we have that $\nu(G(\la))=\nu(\wh{G}(\la)).$ Assume that $n=\nu(G(\la)).$ Thus, $\nu(\wh{G}(\la))=n,$ and $\deg(\det(A_{1}\la+A_0))\geq \nu(\wh{G}(\la))=n.$ However, $\deg(\det(A_{1}\la+A_0))\leq n.$ Therefore, $\deg(\det(A_{1}\la+A_0))=n$ and $\deg(\det(A_{1}\la+A_0))= \nu(\wh{G}(\la)),$ which imply conditions $a)$ and $b),$ respectively. We assume now that conditions $a)$ and $b)$ hold. On the one hand, $A_1$ being invertible implies that $\deg(\det(A_{1}\la+A_0))=n.$ On the other hand, $\mathcal{L}(\la)$ being a minimal polynomial system matrix means that $\deg(\det(A_{1}\la+A_0))=\nu(\wh{G}(\la)).$ Therefore, $ n=\nu(\wh{G}(\la))=\nu(G(\la)).$       
	
\end{proof}

 It is known \cite{strong} that if condition {\em (i)} in Definition \ref{def_stronglin} holds, then condition {\em (ii)} is equivalent to the existence of unimodular matrices $W_1(\la)$ and $W_2(\la)$ such that
 \begin{equation} \label{eq.unimodularoverla}
 W_1(\la)\diag\left(\frac{1}{\la^{g}} G\left(\frac{1}{\la}\right),I_{s_1}\right)
 W_2(\la)=\diag\left(\frac{1}{\la^{\wh{g}}}\wh{G}\left(\frac{1}{\la}\right), I_{s_2}\right).
 \end{equation}
 In Definition \ref{def_stronglin} it can always be taken $s_1=0$ or $s_2=0,$ according to $p\geq q$ and $m\geq r$ or $ q\geq p$ and $r\geq m$. In what follows we will consider  $s_1\geq 0$ and $s_2=0$.  Notice that with this choice and with the notion of reversal given in Definition \ref{reversal}, \eqref{eq.unimodularoverla} is equivalent to
  \begin{equation} \label{eq.unimodularoverla2}
  W_1(\la)\diag\left(\rev G(\la),I_{s_1}\right)
  W_2(\la)=\rev \wh{G}(\la).
  \end{equation}
  
\begin{rem} \rm 
	Notice that Definition \ref{def_stronglin} extends the notion of strong linearization of polynomial matrices in the usual sense \cite[Definition 2.5]{MMMM}. In particular, if $G(\la)$ is a polynomial matrix, then $n=\nu(G(\la))=0.$ Therefore, a strong linearization $\mathcal{L}(\la)$ of $G(\la)$ is of the form $\mathcal{L}(\la)=D_{1}\la + D_0,$ with $\wh{G}(\la)=\mathcal{L}(\la),$ $g=q_1=-\deg (G(\la))$ and $\wh{g}=\wh{q}_{1}=-\deg (\mathcal{L}(\la)).$  
\end{rem}

Let us denote by $\mathcal{N}_r (G(\la))$ and $\mathcal{N}_\ell (G(\la))$ the \textit{right and left null-spaces} over $\FF(\la)$ of $G(\la)$, respectively, i.e.,
if $G(\la)\in\FF(\la)^{p\times m}$,
\[
\begin{array}{l}
\mathcal{N}_r (G(\la))=\{x(\la)\in\FF(\la)^{m\times 1}: G(\la)x(\la)=0\},\\
\mathcal{N}_\ell (G(\la))=\{x(\la)\in\FF(\la)^{p\times 1}: x(\la)^TG(\la)=0\}.

\end{array}
\]
When $G(\lambda)$ is singular at least one of these null spaces is nontrivial. A spectral characterization of strong linearizations is given in \cite[Theorem 6.11]{strong}. It says that $\mathcal{L}(\la)$ is a strong linearization of $G(\la)$ if and only if $\mbox{\rm dim} \, \mathcal{N}_r (G(\la)) = \mbox{\rm dim} \,  \mathcal{N}_r ( \mathcal{L}(\la) )$  and $\mathcal{L}(\la)$ preserves the finite and infinite structures of poles and zeros of $G(\la)$ in the sense of \cite[Definition 6.10]{strong}. This characterization is the key property of strong linearizations in the realm of REPs.
\begin{rem} \rm The equality $\mbox{\rm dim} \, \mathcal{N}_r (G(\la)) = \mbox{\rm dim} \,  \mathcal{N}_r (\mathcal{L}(\la))$
 	is equivalent to $\mbox{\rm dim} \, \mathcal{N}_\ell (G(\la)) = \mbox{\rm dim} \,  \mathcal{N}_\ell (\mathcal{L}(\la)).$ Consider $G(\la)\in\FF(\la)^{p\times m}$ and $\mathcal{L}(\la)\in\FF(\la)^{(p+q)\times (m+q)}$ with $q\geq 0.$  By the rank-nullity theorem $\dim \mathcal{N}_\ell (G(\la))= p-\rank G(\la)$ and
 	$\dim \mathcal{N}_r (G(\la))=m-\rank G(\la)$. Therefore 
 	$\rank \mathcal{L}(\la)=q+\rank G(\la)$ if and only if
 	$\mbox{\rm dim} \, \mathcal{N}_r (G(\la)) = \mbox{\rm dim} \,  \mathcal{N}_r (\mathcal{L}(\la) ).$ And $\rank \mathcal{L}(\la)=q+\rank G(\la)$ if and only if
 	$\mbox{\rm dim} \, \mathcal{N}_\ell (G(\la)) = \mbox{\rm dim} \,  \mathcal{N}_\ell (\mathcal{L}(\la))$.
 	
 \end{rem}

Lemma \ref{more} follows from Definition \ref{def_stronglin}. It shows an easy way to obtain strong linearizations for a rational matrix $G(\lambda)$ from a particular strong linearization $\mathcal{L}(\lambda)$ by multiplying $\mathcal{L}(\lambda)$ by some appropriate matrices. This simple result is fundamental in this paper, and we conjecture that it will be fundamental for constructing (in the future) other families of strong linearizations of rational matrices.

\begin{lem}\label{more} Let $G(\lambda)\in\F(\lambda)^{p\times m}$ be a rational matrix, and let $$\mathcal{L}_{1}(\lambda)=\left[\begin{array}{cc}
	A_{1}\lambda + A_{0} & B_{1}\lambda + B_{0} \\
	-(C_{1}\lambda + C_{0}) & D_{1}\lambda + D_{0}
	\end{array}\right]\in\F[\la]^{(n+(p+s))\times (n+(m+s))} $$ be a strong linearization of $G(\lambda).$ Consider $Q_{1},Q_{3}\in\F^{n\times n},$ $Q_{2}\in\F^{(p+s)\times (p+s)},$ $Q_{4}\in\F^{(m+s)\times (m+s)}$ nonsingular matrices, $W\in\F^{(p+s)\times n},$ and $Z\in\F^{n\times (m+s)}.$ Then the linear polynomial matrix
	$$\mathcal{L}_{2}(\lambda)=\left[ \begin{array}{cc}
	Q_{1} & 0 \\
	W & Q_{2}
	\end{array} \right]\mathcal{L}_{1}(\lambda)\left[ \begin{array}{cc}
	Q_{3} & Z \\
	0 & Q_{4}
	\end{array} \right]$$ is a strong linearization of $G(\lambda).$

\end{lem}
\begin{proof} Let us write
	$$\mathcal{L}_{2}(\lambda)=\left[\begin{array}{cc}
	A_{2}\lambda + \tilde{A}_{0} & B_{2}\lambda + \tilde{B}_{0} \\
	-(C_{2}\lambda + \tilde{C}_{0}) & D_{2}\lambda + \tilde{D}_{0}
	\end{array}\right].$$
	We have $\det (A_{2}\lambda + \tilde{A}_{0})\neq 0$ if $n>0,$ since $A_{2}\lambda + \tilde{A}_{0}= Q_{1}(A_{1}\lambda + A_{0})Q_{3}.$ Let us consider the transfer functions $\wh{G}_{1}(\lambda),$ $\wh{G}_{2}(\lambda)$ of $\mathcal{L}_{1}(\lambda),$ $\mathcal{L}_{2}(\lambda),$ respectively. They satisfy $\wh{G}_{2}(\lambda)=Q_{2}\wh{G}_{1}(\lambda)Q_{4}.$ Let $q_1$ be the first invariant order at infinity of $G(\la)$ and $g=\min(0,q_1)$. For $i=1,2,$ let $\wh{g}_i=\min(0,\wh{q}_i),$ where $\wh{q}_{i}$ is the first invariant order at infinity of $\wh{G}_{i}(\la).$ Since $\mathcal{L}_{1}(\lambda)$ is a strong linearization of $G(\lambda),$ there exist unimodular matrices $U_1(\la)$ and
	$U_2(\la)$ such that $U_1(\la)\diag(G(\la),I_{s})U_2(\la)=\wh{G}_1(\la),$ and biproper matrices $B_1(\la)$ and $B_2(\la)$ such that  
	$B_1(\la)\diag(\la^{g}G(\la),I_{s})B_2(\la)=\la^{\wh{g}_1}\wh{G}_1(\la).$ By using the equality $\wh{G}_{2}(\lambda)=Q_{2}\wh{G}_{1}(\lambda)Q_{4},$ we have that $\wh{g}_1=\wh{g}_2,$ and by the same equality, we get $$Q_2 U_1(\la)\diag(G(\la),I_{s})U_2(\la)Q_4=\wh{G}_2(\la),$$ and $$Q_2B_1(\la)\diag(\la^{g}G(\la),I_{s})B_2(\la)Q_4=\la^{\wh{g}_2}\wh{G}_2(\la).$$ 
	Then we obtain that conditions $(a)$ and $(b)$ in Definition \ref{def_stronglin} hold for $\mathcal{L}_{2}(\lambda).$
\end{proof}

Strong linearizations of a rational matrix $G(\lambda)$ expressed in the form \eqref{eq.polspdec} can be constructed from combining minimal state-space realizations of the strictly proper matrix $G_{sp}(\lambda)$ and strong linearizations of its polynomial part $D(\lambda).$ In particular, strong block minimal bases pencils associated to $D(\lambda)$ with sharp degree can be used (see \cite{strong}). The definition of this concept is taken from \cite{BKL} and appears also in \cite{strong}. As in \cite{BKL}, we will say that a polynomial matrix $K(\la)\in\F[\la]^{p\times m}$ (with $p<m$) is a \textit{minimal basis} if its rows form a minimal basis of the rational subspace they span (see \cite{forney}). Moreover, a minimal basis $N(\la)\in\F[\la]^{q\times m}$ is said to be \textit{dual} to $K(\la)$ if $p+q=m$ and $K(\la)N(\la)^{T}=0$ (see \cite[Definition 2.5]{BKL}).

\begin{deff}\textbf{\cite[Definition 8.1]{strong}} \label{def:minlinearizations} Let $D(\la) \in \FF[\la]^{p \times m}$ be a polynomial matrix. A strong block minimal bases pencil associated to $D(\la)$ is a linear polynomial matrix with the following structure
	\begin{equation}
	\label{eq:minbaspencil}
	\begin{array}{cl}
	\mathcal{L}(\la) =
	\left[
	\begin{array}{cc}
	M(\la) & K_2 (\la)^T \\
	K_1 (\lambda) &0
	\end{array}
	\right]&
	\begin{array}{l}
	\left. \vphantom{K_2 (\la)^T} \right\} {\scriptstyle p + \widehat{p}}\\
	\left. \vphantom{K_1 (\la)} \right\} {\scriptstyle \widehat{m}}
	\end{array}\\
	\hphantom{\mathcal{L}(\la) =}
	\begin{array}{cc}
	\underbrace{\hphantom{K_1 (\lambda)}}_{\scriptstyle m + \widehat{m}} & \underbrace{\hphantom{K_2 (\la)^T}}_{\widehat{p}}
	\end{array}
	\end{array}
	\>,
	\end{equation}
	where $K_1(\la) \in \FF[\la]^{\widehat{m} \times (m + \widehat{m})}$ (respectively $K_2(\la) \in \FF[\la]^{\widehat{p} \times (p + \widehat{p})}$) is a minimal basis with all its row degrees equal to $1$ and with the row degrees of a minimal basis $N_1(\la) \in \FF[\la]^{m \times (m + \widehat{m})}$ (respectively $N_2(\la) \in \FF[\la]^{p \times (p + \widehat{p})}$) dual to $K_1(\la)$ (respectively $K_2(\la)$) all equal, and such that
	\begin{equation} \label{eq:Dpolinminbaslin}
	D(\la) = N_2(\la) M(\la) N_1(\la)^T.
	\end{equation}
	If, in addition, $\deg(D(\la)) = \deg(N_2(\la)) +  \deg(N_1(\la)) + 1$ then $\mathcal{L}(\la)$ is said to be a strong block minimal bases pencil associated to $D(\la)$ with sharp degree.
\end{deff}

\begin{rem} \rm
	The following useful characterization of minimal bases will be used (see \cite[Main Theorem]{forney} or \cite[Theorem 2.2]{BKL}). Namely, $K(\la)\in\F[\la]^{p\times m}$ is a minimal basis if and only if $K(\la_{0})$ has full row rank for all $\la_{0}\in\overline{\F}$ and $K(\la)$ is \textit{row reduced}, i.e., its highest row degree coefficient matrix has full row rank (see \cite[Definition 2.1]{BKL}).
\end{rem}

\begin{rem} \rm A first application of the key Lemma \ref{more} is to construct strong linearizations of a rational matrix $G(\la)$ from any Fiedler-like strong linearization $L_F(\la)$ of its polynomial part $D(\la).$ For this purpose, note that \cite[Theorems 3.8, 3.15, 3.16]{fiedler} guarantee that there exist permutation matrices $\Pi_1$ and $\Pi_2$ and a strong block minimal bases pencil $L(\la)$ associated to $D(\la)$ such that $L_F(\la)=\Pi_1 L(\la)\Pi_2.$ In addition, Theorem 8.11 in \cite{strong} explains how to construct a strong linearization $\mathcal{L}(\la)$ of $G(\la)$ from $L(\la).$ Thus, according to Lemma 2.7, $\diag (I_n,\Pi_1) \mathcal{L}(\la) \diag (I_n,\Pi_2)$ is a strong linearization of $G(\la)$ based on $L_F(\la).$
\end{rem}

In what follows, the Kronecker product of two matrices $A$ and $B$, denoted by $A\otimes B ,$ will be used (see \cite[Chapter 4]{HyJ}).
	
\section{$\M_{1}$-strong linearizations}\label{sect:m1}
 In this section and in Section \ref{sect:m2} we present strong linearizations of square rational matrices $G(\la)$ with polynomial part $D(\la)$ expressed in an orthogonal basis. More precisely, we consider strong linearizations of $D(\lambda)$ that belong to the ansatz spaces $\mathbb{M}_{1}(D)$ or $\mathbb{M}_{2}(D),$ recently developed by H. Faßbender and P. Saltenberger in \cite{ortho}, and based on them, we construct strong linearizations of $G(\la)$ by using Lemma \ref{more} and the strong linearizations presented in \cite[Section 8.2]{strong}. 
 
 As said in the preliminaries, we consider an arbitrary field $\F$ throughout this paper, although the results in \cite{ortho} are stated only for the real field $\R.$ Nevertheless, the results of \cite{ortho} that are used in this paper are also valid for any field $\F.$ We consider a polynomial basis $\{\phi_{j}(\lambda)\}_{j=0}^{\infty}$ of $\F[\lambda],$ viewed as an $\F$-vector space, with $\phi_{j}(\lambda)$ a polynomial of degree $j,$  that satisfies the following three-term recurrence relation:

\begin{equation}\label{recu}
\alpha_{j}\phi_{j+1}(\lambda)=(\lambda-\beta_{j})\phi_{j}(\lambda)-\gamma_{j}\phi_{j-1}(\lambda) \quad j\geq 0
\end{equation}
where $\alpha_{j},\beta_{j},\gamma_{j}\in\F,$ $ \alpha_{j}\neq 0,$ $\phi_{-1}(\lambda)=0,$ and $\phi_{0}(\lambda)=1.$ Let $P(\lambda)\in\F[\lambda]^{m\times m}$ be a polynomial matrix of degree $k$ written in terms of this basis as follows

\begin{equation}\label{polynomial}
P(\lambda)=P_{k}\phi_{k}(\lambda)+P_{k-1}\phi_{k-1}(\lambda)+\cdots + P_{1}\phi_{1}(\lambda)+P_{0}\phi_{0}(\lambda).
\end{equation}
\\
We define $\Phi_{k}(\lambda)=[\phi_{k-1}(\lambda)\cdots \phi_{1}(\lambda)\text{ }\phi_{0}(\lambda)]^{T}$ and $V_{P}=\{v\otimes P(\lambda):v\in\F^{k}\},$ and we consider the set of pencils
$$\mathbb{M}_{1}(P)=\{L(\lambda)=\lambda X+Y: X,Y\in\F^{km\times km},\text{ }L(\lambda)(\Phi_{k}(\lambda)\otimes I_{m})\in V_{P}\}.$$ A pencil $L(\lambda)\in\mathbb{M}_{1}(P),$ which verifies $L(\lambda)(\Phi_{k}(\lambda)\otimes I_{m})=v\otimes P(\lambda)$ for some vector $v\in\F^{k},$ is said to have \textit{right ansatz vector} $v.$ A particular pencil in $\mathbb{M}_{1}(P)$ introduced in \cite[page 63]{ortho} is 

\begin{equation}\label{efe}
F_{\Phi}^{P}(\lambda)=\left[ {\begin{array}{cc}
	m_{\Phi}^{P}(\lambda) \\
	M_{\Phi}(\lambda)\otimes I_{m} \\
	\end{array} } \right]\in \F[\lambda]^{km\times km},
\end{equation}
where
\begin{equation*}\label{min}
m_{\Phi}^{P}(\lambda)=\left[ \dfrac{(\lambda-\beta_{k-1})}{\alpha_{k-1}}P_{k}+P_{k-1}\quad P_{k-2}-\dfrac{\gamma_{k-1}}{\alpha_{k-1}}P_{k}\quad P_{k-3}\quad \cdots\quad P_{1}\quad P_{0}\right],
\end{equation*}
and
\begin{equation*}\label{may}
M_{\Phi}(\lambda)=
\left[ {\begin{array}{cccccc}
	-\alpha_{k-2} & (\lambda -\beta_{k-2}) & -\gamma_{k-2} &  \\
	& -\alpha_{k-3} & (\lambda - \beta_{k-3}) & -\gamma_{k-3} &  \\
	& &\ddots &\ddots & \ddots &  \\
	& & &-\alpha_{1}& (\lambda-\beta_{1}) & -\gamma_{1} \\
	& & & &-\alpha_{0}& (\lambda-\beta_{0})
	\end{array} } \right] .
\end{equation*}
Since $m_{\Phi}^{P}(\lambda)(\Phi_{k}(\lambda)\otimes I_{m})=P(\lambda)$ and $(M_{\Phi}(\lambda)\otimes I_{m})(\Phi_{k}(\lambda)\otimes I_{m})=0,$ we get that $F_{\Phi}^{P}(\lambda)(\Phi_{k}(\lambda) \otimes I_{m})=e_{1}\otimes P(\lambda),$ where $e_1$ is the first canonical vector of $\F^{k}.$ Therefore, $F_{\Phi}^{P}(\lambda)\in \mathbb{M}_{1}(P)$ with right ansatz vector $e_{1}\in \F^{k}.$ This particular example is very important because, by using it, we can obtain all the elements in $\mathbb{M}_{1}(P).$ This follows from the next theorem.

\begin{theo}\textbf{\cite[Theorem 1]{ortho}}
	Let $P(\lambda)\in\F[\lambda]^{m\times m }$ be a polynomial matrix with degree $k\geq 2.$ Then $L(\lambda)\in\mathbb{M}_{1}(P)$ with right ansatz vector $v\in \F^{k}$ if and only if $$L(\lambda)=[v\otimes I_{m}\quad H]F_{\Phi}^{P}(\lambda)$$
	for some matrix $H\in\F^{km \times (k-1)m.}$
\end{theo}

\begin{rem} \rm
	For the monomial basis $\{\phi_{j}(\lambda)=\lambda^{j}\}_{j=0}^{\infty}$ the space $\M_{1}(P)$ is denoted $\mathbb{L}_{1}(P)$ (see \cite{MMMM}). In this case $\alpha_{j}=1$ and $\beta_{j}=\gamma_{j}=0$ for all $j\geq 0$ in \eqref{recu} and the matrix $F_{\Phi}^{P}(\lambda)$ is the first companion form of $P(\lambda).$ 
\end{rem}

It is known that $F_{\Phi}^{P}(\lambda)$ is a strong linearization of $P(\lambda)$ (see \cite[Theorem 2]{polybases} for regular polynomial matrices $P(\la)$, and \cite[Section 7]{singular} for singular), but we can obtain this property as an immediate corollary of the next result.

\begin{lem}\label{strongblock} $F_{\Phi}^{P}(\lambda)$ is a strong block minimal bases pencil with only one block column associated to $P(\lambda)$ with sharp degree. Moreover, $\Phi_{k}(\lambda)^{T}\otimes I_{m}$ is a minimal basis dual to the minimal basis $M_{\Phi}(\lambda)\otimes I_{m}.$
\end{lem}

\begin{proof} Let us denote $M(\lambda)=m_{\Phi}^{P}(\lambda)$ and $K(\lambda)=M_{\Phi}(\lambda)\otimes I_{m}.$ We consider
	$$F_{\Phi}^{P}(\lambda)=\left[ {\begin{array}{cc}
		M(\lambda) \\
		K(\lambda) \\
		\end{array} } \right].$$
Note that $M_{\Phi}(\lambda_{0})$ has full row rank for all $\lambda_{0}\in \overline{\efe}$ because $\alpha_{i}\neq 0$ for all $i\geq 0.$ Also note that $M_{\Phi}(\lambda)$ is row reduced because its highest row degree coefficient matrix $$[M_{\Phi}]_{hr}=\left[ {\begin{array}{cccccc}
		0 & 1 & 0 &  \\
		& 0 & 1 & 0 &  \\
		& &\ddots &\ddots & \ddots &  \\
		& & &0& 1 & 0 \\
		& & & & 0 & 1
		\end{array} } \right] $$
	has full row rank. We conclude that $M_{\Phi}(\lambda)$ is a minimal basis, and therefore, $K(\lambda)=M_{\Phi}(\lambda)\otimes I_{m}$ is also a minimal basis \cite[Corollary 2.4]{BKL}.
	Let us denote $N(\lambda)= \Phi_{k}(\lambda)^{T}\otimes I_{m}.$ Note that $\Phi_{k}(\lambda)^{T}$ is a minimal basis because $\phi_{0}(\lambda)=1,$ so $\Phi_{k}(\lambda_{0})$ has rank $1$ for all $\lambda_{0}\in\overline{\efe},$ and
	$$[\Phi_{k}^{T}]_{hr}=\left[ {\begin{array}{cccc}
		\frac{1}{\alpha_{0}\alpha_{1}\cdots \alpha_{k-2} }  &
		0   &
		\cdots  &
		0
		\end{array} } \right] $$
	has also rank $1.$ Therefore, $N(\lambda)=\Phi_{k}(\lambda)^{T}\otimes I_{m}$ is also a minimal basis.
	Since $K(\lambda)N(\lambda)^{T}=(M_{\Phi}(\lambda)\otimes I_{m})(\Phi_{k}(\lambda)\otimes I_{m})=0$ and $\left[ {\begin{array}{cc}
		K(\lambda) \\
		N(\lambda) \\
		\end{array} } \right]$ is a square matrix, we have that $K(\lambda)$ and $N(\lambda)$ are dual minimal bases. In addition, it is obvious that all the row degrees of $K(\lambda)$ are equal to $1$ and all the row degrees of $\Phi_{k}(\lambda)^{T}\otimes I_{m}$ are equal to $k-1.$ Hence, $F_{\Phi}^{P}(\lambda)$ is a strong block minimal bases pencil associated to the polynomial matrix $M(\lambda)N(\lambda)^{T}=m_{\Phi}^{P}(\lambda) (\Phi_{k}(\lambda)\otimes I_{m})=P(\lambda)$
and $\text{deg}(P(\lambda))=1+\text{deg}(N(\lambda)),$ which means that $F_{\Phi}^{P}(\lambda)$ has sharp degree.
\end{proof}
Since every strong block minimal bases pencil is a strong linearization (see \cite[Theorem 3.3]{BKL}), the following corollary is straightforward.

\begin{con}
	$F_{\Phi}^{P}(\lambda)$ is a strong linearization for $P(\lambda).$
\end{con}

The proof of the next result is trivial because if $ L(\lambda) = [v\otimes I_{m} \quad H]F_{\Phi}^{P}(\lambda) $ with $[v\otimes I_{m} \quad H] $ nonsingular then $L(\lambda)$ is strictly equivalent to $F_{\Phi}^{P}(\lambda).$

\begin{con}\textbf{\cite[Corollary 2.1]{ortho}}\label{ele}
	Let $L(\lambda)=[v\otimes I_{m}\quad H]F_{\Phi}^{P}(\lambda)\in\mathbb{M}_{1}(P).$ If $[v\otimes I_{m} \quad H]$ is nonsingular then $L(\lambda)$ is a strong linearization for $P(\lambda).$
\end{con}

\begin{rem}\rm Although $F_{\Phi}^{P}(\lambda)$ is a strong block minimal bases pencil associated to $P(\lambda)$ this structure is not preserved in general when we multiply on the left by a nonsingular matrix $[v\otimes I_{m} \quad H].$ For example, consider the polynomial matrix $P(\lambda)=I\lambda^{3}+2I\lambda^{2}  + I\lambda+ S \in\R[\la]^{2\times 2}$ expressed in the monomial basis, where $S=\left[\begin{array}{cc}
	1 & 0 \\
	0 & 0
	\end{array}\right]$ and $I$ stands for $I_2.$ In this case, the matrix $F_{\Phi}^{P}(\lambda)$ is $F_{\Phi}^{P}(\lambda)=\left[\begin{array}{ccc}
\lambda I + 2I & I & S \\
-I & \lambda I & 0 \\
0 & -I & \lambda I
\end{array} \right].$ Let $v=\left[1\quad 1\quad 0\right]^{T}$ and $ H =\left[\begin{array}{cc}
0 & 0 \\
I & 0 \\
0 & I
\end{array} \right]$ and let $L(\lambda)=[v\otimes I \quad H]F_{\Phi}^{P}(\lambda)=\left[\begin{array}{ccc}
\lambda I + 2I & I & S \\
\la I + I & \lambda I + I & S \\
0 & -I & \lambda I
\end{array} \right].$ Notice that if $L(\lambda)$ were a strong block minimal bases pencil associated to $P(\lambda),$ one of these two different situations would happen in \eqref{eq:minbaspencil}:
\begin{enumerate}
	\item $M(\lambda)=\left[\lambda I + 2I \quad I \quad S \right],$ $K_{1}(\lambda)=\left[\begin{array}{ccc}
\la I + I & \lambda I + I & S \\
0 & -I & \lambda I
	\end{array}\right]$ and $K_{2}(\lambda)$ empty. 
	\item $M(\lambda)=\left[\begin{array}{ccc}
	\lambda I + 2I  \\
	\la I + I  \\
	0 
	\end{array} \right],$ $K_{2}(\lambda)^{T}=\left[\begin{array}{ccc}
	 I & S \\
	 \lambda I + I & S \\
	-I & \lambda I
	\end{array} \right]$ and $K_{1}(\lambda)$ empty.  
\end{enumerate}   
  In the first case, the matrix $K_{1}(\lambda)$ has not full row rank for $\lambda=-1.$ In the second case, the matrix $K_{2}(\la)$ has not full row rank for $\lambda=0.$ Therefore, $L(\la)$ is not a strong block minimal bases pencil associated to $P(\lambda).$ One still may wonder whether or not the pencil we obtain by permuting the first and third columns of $L(\la)$ would be a strong block minimal bases pencil of $P(\la),$ since this pencil has a zero block in the right-lower corner. Observe that this cannot happen because the polynomial associated to such pencil would have size $4\times 4.$
\end{rem}

From the fact that $F_{\Phi}^{P}(\lambda)$ is a strong block minimal bases pencil, we can obtain strong linearizations for rational matrices by applying Theorem 8.11 in \cite{strong}. For this purpose, we prove first the following lemma. 

\begin{lem}\label{unimod}
The matrix $$U(\lambda)=\left[ {\begin{array}{cc}
	M_{\Phi}(\lambda)\otimes I_{m} \\
	e_{k}^{T}\otimes I_{m}
	\end{array} } \right]=\left[ {\begin{array}{cc}
	M_{\Phi}(\lambda) \\
	e_{k}^{T}
	\end{array} } \right]\otimes I_{m}$$
is unimodular, and its inverse has the form $U(\lambda)^{-1}=[\widehat{\Phi}_{k}(\lambda)\quad \Phi_{k}(\lambda)\otimes I_{m}]$
with $\widehat{\Phi}_{k}(\lambda)\in\F[\lambda]^{km\times (k-1)m}.$
\end{lem}

\begin{proof} Let us consider the matrix $$\tilde{U}(\lambda)=\left[ {\begin{array}{cc}
		M_{\Phi}(\lambda) \\
		e_{k}^{T}
		\end{array} } \right]=\left[ {\begin{array}{cccccc}
		-\alpha_{k-2} & (\lambda -\beta_{k-2}) & -\gamma_{k-2} &  \\
		& -\alpha_{k-3} & (\lambda - \beta_{k-3}) & -\gamma_{k-3} &  \\
		& &\ddots &\ddots & \ddots &  \\
		& & &-\alpha_{1}& (\lambda-\beta_{1}) & -\gamma_{1} \\
		& & & &-\alpha_{0}& (\lambda-\beta_{0})\\
		0 & & \cdots & & 0 & 1
		\end{array} } \right] .$$
Since $\tilde{U}(\lambda)$ is upper triangular, its determinant is $(-\alpha_{k-2})\cdots (-\alpha_{0}),$ i.e., a constant different from zero. Therefore, $\tilde{U}(\lambda)$ is unimodular. Finally, note that $\tilde{U}(\lambda)\Phi_{k}(\lambda)=e_{k}\in\F^{k}.$ Thus
$\Phi_{k}(\lambda)$ is the last column of $\tilde{U}(\lambda)^{-1}.$ 
\end{proof}

\begin{theo}\label{prim} Let $G(\lambda)\in \F(\lambda)^{m\times m}$ be a rational matrix, let $G(\lambda)=D(\lambda)+G_{sp}(\lambda)$ be its unique decomposition into its polynomial part $D(\lambda)\in\F[\lambda]^{m\times m}$ and its stricly proper part $G_{sp}(\lambda)\in\F(\lambda)^{m\times m},$ and let $G_{sp}(\lambda)=C(\lambda I_{n}-A)^{-1}B$ be a minimal order state-space realization of $G_{sp}(\lambda).$ Assume that $deg(D(\lambda))\geq 2.$ Write $D(\lambda)$ in terms of the polynomial basis $\{\phi_{j}(\lambda)\}_{j=0}^{\infty}$ satisfying the three-term recurrence relation \eqref{recu}, as
	\begin{equation}\label{polypart_polybasis}
D(\lambda)=D_{k}\phi_{k}(\lambda)+D_{k-1}\phi_{k-1}(\lambda)+\cdots + D_{1}\phi_{1}(\lambda)+D_{0}\phi_{0}(\lambda)
	\end{equation}
	with $D_{k}\neq 0,$ and let $F_{\Phi}^{D}(\lambda)$ be the matrix pencil defined as in \eqref{efe}. Then, for any nonsingular matrices $X,Y\in\F^{n\times n}$ the linear polynomial matrix
	$$\mathcal{L}(\lambda)= \left[
		\begin{array}{c|c}
		
		X(\lambda I_{n}-A)Y& 0_{n\times (k-1)m}\quad XB\\
		\hline \phantom{\Big|}
		
		\begin{array}{c}
		-CY\\
		0_{(k-1)m\times n}
		\end{array}& F_{\Phi}^{D}(\lambda)
		\end{array}
		\right]$$
	is a strong linearization of $G(\lambda).$
	
\end{theo}

\begin{proof}
Lemmas \ref{strongblock} and \ref{unimod} allow us to apply \cite[Theorem 8.11]{strong}, with $K_{1}(\lambda)= M_{\Phi}(\lambda)\otimes I_{m} ,$ $\widehat{K}_{1}=e_{k}^{T}\otimes I_{m},$ $K_{2}(\lambda)^{T}$ empty and $\widehat{K}_{2}^{T}=I_{m}.$
\end{proof}

Then, from combining Lemma \ref{more} and Corollary \ref{ele} we obtain strong linearizations of a rational matrix from strong linearizations in $\M_{1}(D)$ of its polynomial part.

\begin{theo}\label{muno}  Under the same assumptions as in Theorem \ref{prim}, let $v\in\F^{k},$ $H\in\F^{km\times (k-1)m}$ with $[v\otimes I_{m}\quad H]$ nonsingular and let $L(\lambda)=[v\otimes I_{m}\quad H]F_{\Phi}^{D}(\lambda)\in \M_{1}(D).$ Then, the linear polynomial matrix
	$$\mathcal{L}(\lambda)= \left[
	\begin{array}{c|c}
	
	X(\lambda I_{n}-A)Y& 0_{n\times (k-1)m}\quad XB\\
	\hline \phantom{\Big|}
	
	-(v\otimes I_{m})CY& L(\lambda)
	\end{array}
	\right]$$
	is a strong linearization of $G(\lambda).$
	
\end{theo}

\begin{proof}
	 Set $K=[v\otimes I_{m}\quad H].$ If $K$ is nonsingular then, by Lemma \ref{more} and Theorem \ref{prim},
	
\begin{eqnarray*}
	\mathcal{L}(\lambda) &= &\left[\begin{array}{cc}
	I_{n} & 0 \\
	0 & K
	\end{array}\right]\left[
	\begin{array}{c|c}
	
	X(\lambda I_{n}-A)Y& 0_{n\times (k-1)m}\quad XB\\
	\hline \phantom{\Big|}
	
	\begin{array}{c}
	-CY\\
	0_{(k-1)m\times n}
	\end{array}& F_{\Phi}^{D}(\lambda)
	\end{array}
	\right]\\
	&=&\left[
	\begin{array}{c|c}
	
	X(\lambda I_{n}-A)Y& 0_{n\times (k-1)m}\quad XB\\
	\hline \phantom{\Big|}
	
	-(v\otimes I_{m})CY& L(\lambda)
	\end{array}
	\right].
\end{eqnarray*} 
is a strong linearization of $G(\lambda).$ 
\end{proof}  
The strong linearizations of square rational matrices constructed in Theorem \ref{muno} will be called \textit{$\M_{1}$-strong linearizations.}

\section{$\M_{2}$-strong linearizations}\label{sect:m2}

In this section we obtain strong linearizations of a square rational matrix from the transposed version of $\mathbb{M}_{1}(P),$ where $P(\lambda)$ is the polynomial matrix in \eqref{polynomial}. Since the proofs of the results are similar to those in Section \ref{sect:m1}, they are omitted for brevity. We define $W_{P}=\{w^{T}\otimes P(\lambda):w\in\F^{k}\},$ and we consider the set of pencils
$$\mathbb{M}_{2}(P)=\{L(\lambda)=\lambda X+Y: X,Y\in\F^{km\times km},\text{ }(\Phi_{k}(\lambda)^{T}\otimes I_{m})L(\lambda)\in W_{P}\}.$$
A pencil $L(\lambda)\in\mathbb{M}_{2}(P),$ which verifies $(\Phi_{k}(\lambda)^{T}\otimes I_{m})L(\lambda)=w^{T}\otimes P(\lambda)$ for some vector $w\in\F^{k},$ is said to have \textit{left ansatz vector} $w.$ Pencils in $\mathbb{M}_2(P)$ are characterized in \cite[Theorem 2]{ortho}. We need the definition of the block-transpose of a $km\times lm$ pencil $L(\lambda).$ If we express $L(\lambda)$ as
$L(\lambda)=\displaystyle\sum_{i=1}^{k}\displaystyle\sum_{j=1}^{l}e_{i}e_{j}^{T}\otimes L_{ij}(\lambda)$ for certain $m\times m$ pencils $L_{ij}(\lambda),$ where $e_i$ denotes the $i$th canonical vector in $\F^{k},$ and $e_j$ the $j$th canonical vector in $\F^{l},$ we call $L(\lambda)^{\mathcal{B}}=\displaystyle\sum_{i=1}^{k}\displaystyle\sum_{j=1}^{l}e_{j}e_{i}^{T}\otimes L_{ij}(\lambda)$ the \textit{block-transpose} of $L(\lambda).$ Notice that $F_{\Phi}^{P}(\lambda)^{\mathcal{B}}=[
	{m_{\Phi}^{P}(\lambda)}^{\mathcal{B}} \quad
	{M_{\Phi}(\lambda)}^{T}\otimes I_{m} ].$

\begin{theo} \textbf{\cite[Theorem 2]{ortho}}
	Let $P(\lambda)\in\F[\lambda]^{m\times m }$ be a polynomial matrix with degree $k\geq 2.$ Then $L(\lambda)\in\mathbb{M}_{2}(P)$ with left ansatz vector $w\in \F^{k}$ if and only if $$L(\lambda)=F_{\Phi}^{P}(\lambda)^{\mathcal{B}}\left[\begin{array}{c}
     w^{T}\otimes I_{m} \\
     H^{\mathcal{B}}
	\end{array}\right]$$
	for some matrix $H\in\F^{km \times (k-1)m}$ partitioned into $k\times (k-1)$ blocks each of size $m\times m.$
\end{theo}

The vector space $\M_{2}(P)$ reduces to the well-known space $\mathbb{L}_{2}(P)$ when $\{\phi_{k}(\lambda)\}_{k=0}^{\infty}$ is the monomial basis, see \cite{MMMM}. Lemma \ref{strongblockrow} is for $\M_{2}(P)$ the counterpart of Lemma \ref{strongblock} for $\mathbb{M}_{1}(P)$ and can be used to proceed with $\M_{2}(P)$ analogously as we did with $\mathbb{M}_{1}(P).$

\begin{lem}\label{strongblockrow} $F_{\Phi}^{P}(\lambda)^{\mathcal{B}}$ is a strong block minimal bases pencil with only one block row associated to $P(\lambda)$ with sharp degree. 
\end{lem}

 In particular, Lemma \ref{strongblockrow} allows us to apply \cite[Theorem 8.11]{strong} to the strong linearization $F_{\Phi}^{D}(\lambda)^{\mathcal{B}}$ of the polynomial part of a square rational matrix, with $K_{2}(\lambda)=M_{\Phi}(\lambda)\otimes I_{m},$ 
  $\widehat{K}_{2}=e_{k}^{T}\otimes I_{m},$ $K_{1}(\lambda)$ empty and $\widehat{K}_{1}=I_{m}.$ Thus, we get the following results to obtain strong linearizations of a square rational matrix $G(\la)=D(\la)+G_{sp}(\la)$ expressed as in \eqref{eq.polspdec} from strong linearizations in $\mathbb{M}_{2}(D).$

\begin{theo}\label{seg} Under the same assumptions as in Theorem \ref{prim}, the linear polynomial matrix
	$$\mathcal{L}(\lambda)= \left[
	\begin{array}{c|c}
	
	X(\lambda I_{n}-A)Y& XB\quad  0_{n\times (k-1)m}\\
	\hline 
	
	\begin{array}{c}
	0_{(k-1)m\times n}\\
	-CY
	\end{array}& F_{\Phi}^{D}(\lambda)^{\mathcal{B}}
	\end{array}
	\right]$$
	is a strong linearization of $G(\lambda).$
	
\end{theo}

\begin{theo}\label{mdos} Under the same assumptions as in Theorem \ref{prim}, let $w\in\F^{k},$ $H\in\F^{km\times (k-1)m}$ with $\left[\begin{array}{c}
	w^{T}\otimes I_{m} \\
	H^{\mathcal{B}}
	\end{array}\right]$ nonsingular and let $L(\lambda)=F_{\Phi}^{D}(\lambda)^{\mathcal{B}}\left[\begin{array}{c}
	w^{T}\otimes I_{m} \\
	H^{\mathcal{B}}
	\end{array}\right]\in\mathbb{M}_{2}(D).$ Then the linear polynomial matrix
	$$\mathcal{L}(\lambda)= \left[
	\begin{array}{c|c}
	
	X(\lambda I_{n}-A)Y& XB(w^{T}\otimes I_{m})\\
	\hline \phantom{\Big|}
	
		\begin{array}{c}
		0_{(k-1)m\times n}\\
		-CY
		\end{array}& L(\lambda)
	\end{array}
	\right]$$
	is a strong linearization of $G(\lambda).$
	
\end{theo} 

\begin{proof}
	We apply Lemma \ref{more} by multiplying on the right the matrix $ 	\mathcal{L}(\lambda)$ in Theorem \ref{seg} by the matrix $\left[\begin{array}{cc}
	I_{n} & 0 \\
	0 & K
	\end{array}\right]$ with $K=\left[\begin{array}{c}
	w^{T}\otimes I_{m} \\
	H^{\mathcal{B}}
	\end{array}\right]$ nonsingular.
\end{proof}

The strong linearizations of rational matrices constructed in Theorem \ref{mdos} will be called \textit{$\M_{2}$-strong linearizations.}

\section{Recovering eigenvectors from $\mathbb{M}_{1}$- and $\mathbb{M}_{2}$- strong linearizations of rational matrices}\label{recovery}

In this section we will recover right and left eigenvectors of a rational matrix. These eigenvectors will be obtained without essentially computational cost from the right and left eigenvectors of the strong linearizations that we have constructed in Theorems \ref{muno} and \ref{mdos}. Previously, and due to the fact that we can see strong linearizations as polynomial system matrices, in Subsection \ref{eigenvectors} we will see the relation between the eigenvectors of a polynomial system matrix and the eigenvectors of its transfer function matrix. For the sake of brevity, in this section the following nomenclature is adopted: ``$(\la_0,x_0)$ is a solution of the REP $G(\la)x=0$'' means that $\la_0$ is a finite eigenvalue of $G(\la)\in\F(\la)^{p\times m}$ and $x_0$ is a right eigenvector corresponding to $\la_0,$ and ``$(\la_0,x_0)$ is a solution of the REP $x^{T}G(\la)=0$'' means that $\la_0$ is a finite eigenvalue of $G(\la)\in\F(\la)^{p\times m}$ and $x_0$ is a left eigenvector corresponding to $\la_0.$ An analogous notation is adopted for polynomial eigenvalue problems (PEPs) and the particular case of linear eigenvalue problems (LEPs).

\subsection{Eigenvectors of polynomial system and transfer function matrices}\label{eigenvectors}

We know from \cite[Proposition 3.1]{strong} how to recover right eigenvectors of a polynomial system matrix $P(\lambda)$ from those of its transfer function $G(\lambda),$ and conversely. In Proposition \ref{righteigen} we state a extended version of \cite[Proposition 3.1]{strong} that includes a result about the null-spaces of $P(\la)$ and $G(\la)$ evaluated at the eigenvalue of interest. That is, for a finite eigenvalue $\lambda_{0}$ of a rational matrix $G(\lambda)\in\FF(\la)^{p\times m}$ we denote by $\mathcal{N}_r (G(\la_{0}))$ the right null-space over $\overline{\FF}$ of $G(\la_{0})$, i.e., $\mathcal{N}_r (G(\la_{0}))=\{x\in\overline{\FF}^{m\times 1}: G(\la_{0})x=0\}.$ We state without proof the analogous result for left eigenvectors and null-spaces in Proposition \ref{lefteigen}. 

In what follows, we assume that eigenvectors of the form $\begin{bmatrix}y\\x\end{bmatrix}$ are partitioned conformable to the corresponding polynomial system matrix.

\begin{prop} \label{righteigen} Let $G(\la)\in\FF(\la)^{p\times m}$ be a rational matrix and
	\[
	P(\la)=\begin{bmatrix}
	A(\la) & B(\la)\\
	-C(\la) & D(\la)
	\end{bmatrix}\in \FF[\la]^{(n+p)\times (n+m)}
	\]
	be any polynomial system matrix with $G(\la)$ as transfer function matrix. 
	\begin{itemize}
		
		 \item[a)]If $\left(\la_0,\begin{bmatrix}y_0\\x_0\end{bmatrix}\right)$ is a solution of the PEP $P(\la)z=0$ such that $\det A(\la_0)\neq 0$, then $(\la_0,x_0)$ is a solution of the REP $G(\la)x=0$.
		 \item[b)] Moreover, if $\left\{\begin{bmatrix}y_1\\x_1\end{bmatrix},\ldots,\begin{bmatrix}y_t\\x_t\end{bmatrix}\right\}$ is a basis of $\mathcal{N}_{r}(P(\lambda_{0})),$ with $\det A(\la_0)\neq 0$, then $\{x_{1},\dots,x_{t}\}$ is a basis of $\mathcal{N}_{r}(G(\lambda_{0})).$
		 
		\item[c)]Conversely, if $(\la_0,x_0)$ is a solution of the REP $G(\la)x=0$ such that $\det A(\la_0)\neq 0$ and $y_0$ is defined as the
		unique solution of  $A(\la_0)y_0+B(\la_0)x_0=0,$ then
		$\left(\la_0,\begin{bmatrix}y_0\\x_0\end{bmatrix}\right)$ is a
		solution of the PEP $P(\la)z=0$.
		\item[d)]  Moreover, if $\{x_{1},\dots,x_{t}\}$ is a basis of $\mathcal{N}_{r}(G(\lambda_{0})),$ with $\det A(\la_0)\neq 0$, and, for $i=1,\ldots,t,$ $y_i$ is defined as the
		unique solution of  $A(\la_0)y_i+B(\la_0)x_i=0,$ then $\left\{\begin{bmatrix}y_1\\x_1\end{bmatrix},\ldots,\begin{bmatrix}y_t\\x_t\end{bmatrix}\right\}$ is a basis of $\mathcal{N}_{r}(P(\lambda_{0})).$

	\end{itemize}
\end{prop}
\begin{proof} The statements $a)$ and $c)$ are the results in \cite[Proposition 3.1]{strong} stated here for a rectangular matrix $G(\la).$ The proofs are exactly the same as in \cite{strong} and, therefore, are omitted. To prove $b)$ and $d)$ we write 
	$$\begin{bmatrix}
	A(\la_{0}) & B(\la_{0})\\
	-C(\la_{0}) & D(\la_{0})
	\end{bmatrix}=\begin{bmatrix}
	I_n & 0\\
	-C(\la_{0})A(\la_{0})^{-1} & I_p
	\end{bmatrix} 
	\begin{bmatrix}
	A(\lambda_{0}) & 0\\
	0 & G(\lambda_{0})
	\end{bmatrix}
	\begin{bmatrix}
	I_n & A(\la_{0})^{-1}B(\la_{0})\\
	0 & I_m
	\end{bmatrix}.$$
	Since $\det A(\la_0)\neq 0,$ $\text{rank}(P(\la_{0}))=n+\text{rank}(G(\la_{0})).$ Therefore
	\begin{equation}\label{dimequal}
\text{dim}\;\mathcal{N}_{r}(P(\lambda_{0}))=\text{dim}\;\mathcal{N}_{r}(G(\lambda_{0})).
	\end{equation} 
	Then $b)$ and $d)$ are obtained by using $a)$ and $c),$ respectively, taking \eqref{dimequal} and the linear independence of the considered sets into account, and observing that $P(\la_0)\begin{bmatrix}y_0\\x_0\end{bmatrix}=0$ if and only if $y_0=-A(\la_0)^{-1}B(\la_0)x_0$ and $G(\la_0)x_0=0.$ 
\end{proof}

Proposition \ref{lefteigen} is an analogous result to Proposition \ref{righteigen} for left eigenvectors and left null-spaces as we announced, and it can be proved in a similar way. The left null-space of $G(\la_0)\in\overline{\F}^{p\times m}$ is denoted and defined as $\mathcal{N}_\ell (G(\la_{0}))=\{x\in\overline{\F}^{p\times 1}: x^{T}G(\la_{0})=0\}.$

\begin{prop}\label{lefteigen}	Let $G(\lambda)\in\F(\lambda)^{p\times m}$ be a rational matrix and $$P(\lambda)=\left[\begin{array}{cc}
	A(\lambda) & B(\lambda)\\
	-C(\lambda) & D(\lambda)
	\end{array}\right]\in \F[\lambda]^{(n+p)\times(n+m)}$$ be any polynomial system matrix with $G(\lambda)$ as transfer function matrix. 
		\begin{itemize}
			\item[a)] If $\left(\lambda_{0},\begin{bmatrix}y_0\\x_0\end{bmatrix}\right)$  is a solution of the PEP $z^{T}P(\lambda)=0$ such that $\det A(\lambda_{0})\neq 0,$ then $(\lambda_{0},x_{0})$ is a solution of the REP $x^{T}G(\lambda)=0.$

	\item[b)] 	Moreover, if $\left\{\begin{bmatrix}y_1\\x_1\end{bmatrix},\ldots,\begin{bmatrix}y_q\\x_q\end{bmatrix}\right\}$ is a basis of $\mathcal{N}_{\ell}(P(\lambda_{0})),$ with $\det A(\la_0)\neq 0$, then $\{x_{1},\dots,x_{q}\}$ is a basis of $\mathcal{N}_{\ell}(G(\lambda_{0})).$   
	
	\item[c)] Conversely, if $( \lambda_{0},x_{0} )$ is a solution of the REP $x^{T}G(\lambda)=0$ such that $\det A(\lambda_{0})\neq 0,$ and $y_{0}$ is defined as the unique solution of $y_{0}^{T}A(\lambda_{0})-x_{0}^{T}C(\lambda_{0})=0,$ then $\left(\lambda_{0},\begin{bmatrix}y_0\\x_0\end{bmatrix}\right)$ is a solution of the PEP $z^{T}P(\lambda)=0.$ 
	\item[d)] Moreover, if $\{x_{1},\dots,x_{q}\}$ is a basis of $\mathcal{N}_{\ell}(G(\lambda_{0})),$ with $\det A(\la_0)\neq 0$, and, for $i=1,\ldots,q,$ $y_i$ is defined as the
	unique solution of  $y_{i}^{T}A(\la_0)-x_{i}^{T} C( \la_{0})=0,$ then $\left\{\begin{bmatrix}y_1\\x_1\end{bmatrix},\ldots,\begin{bmatrix}y_q\\x_q\end{bmatrix}\right\}$ is a basis of $\mathcal{N}_{\ell}(P(\lambda_{0})).$

\end{itemize}
\end{prop}

\begin{rem}\label{singularcase1} \rm If $G(\lambda)$ is singular, then for any $\la_0\in\overline{\F}$ that is not a pole of $G(\lambda),$ including those $\la_0$ that are not eigenvalues of $G(\la),$ $\mathcal{N}_{r}(G(\lambda_{0}))\neq \{0\}$ or $\mathcal{N}_{\ell}(G(\lambda_{0}))\neq \{0\}.$ The reader can check easily that Propositions \ref{righteigen} and \ref{lefteigen} remain valid for any $\la_0\in\overline{\F}$ that is not a pole of $G(\lambda)$ in the case $G(\la)$ is singular.
\end{rem}

\subsection{Eigenvectors from $\M_{1}$-strong linearizations}\label{eigenvectorfromm1}
We consider in this subsection the linearizations that we have constructed in Theorem \ref{muno}, which we called $\M_{1}$-strong linearizations. We will recover the eigenvectors of a rational matrix $G(\lambda)$ from those of its $\M_{1}$-strong linearizations, and conversely. Lemma \ref{lemmaright} will be used for this purpose.
\begin{lem}\label{lemmaright} Let $G(\lambda)\in\FF(\la)^{m\times m}$ be a rational matrix with polynomial part of degree $k\geq 2,$ let $$\mathcal{L}(\lambda)= \left[
	\begin{array}{c|c}
	
	X(\lambda I_{n}-A)Y& 0_{n\times (k-1)m}\quad XB\\
	\hline \phantom{\Big|}
	
	-(v\otimes I_{m})CY& L(\lambda)
	\end{array}
	\right]$$
	be an $\M_{1}$-strong linearization of $G(\lambda),$ and let $\wh{G}(\lambda)$ be the transfer function of $\mathcal{L}(\lambda).$ Then 
	\begin{equation}\label{transfer}
	\wh{G}(\lambda)(\Phi_{k}(\lambda)\otimes I_{m})=v\otimes G(\lambda).
	\end{equation}
	
\end{lem}
\begin{proof} We consider the transfer function of the matrix $ \mathcal{L}(\lambda),$
	$$\wh{G}(\lambda)= L(\lambda) + \left[0_{km\times (k-1)m}\quad (v\otimes I_{m})C(\lambda I_{n} - A)^{-1}B\right].$$ Let $D(\la)$ be the polynomial part of $G(\la)$. Since $L(\lambda)$ belongs to $\M_{1}(D),$ $L(\lambda)(\Phi_{k}(\lambda)\otimes I_m)=v\otimes D(\lambda)=(v\otimes I_m)D(\lambda).$ Therefore, we obtain  	
	\begin{equation*}
	\begin{split}
	\wh{G}(\lambda)(\Phi_{k}(\lambda)\otimes I_{m}) & = (L(\lambda) + \left[0_{km\times (k-1)m}\quad (v\otimes I_{m})C(\lambda I_{n} - A)^{-1}B\right])(\Phi_{k}(\lambda)\otimes I_{m}) \\
	& = (v\otimes I_{m})D(\lambda) + (v\otimes I_{m}) C(\lambda I_{n} - A)^{-1}B \\
	& = (v\otimes I_{m})G(\lambda).
	\end{split}
	\end{equation*}
\end{proof}
	
	\begin{rem}\label{dimensiones} \rm 
Since $\mathcal{L}(\lambda)$ is a strong linearization of the rational matrix $G(\lambda)$ we have, by Definition \ref{def_stronglin}, that there are unimodular matrices $U(\lambda),V(\lambda)\in\F[\lambda]^{km\times km}$ such that
\begin{equation}\label{unimodulareq}
U(\lambda)\wh{G}(\lambda)V(\lambda)=\diag(G(\lambda),I_{(k-1)m}).
\end{equation}
Thus, if we consider a finite eigenvalue $\lambda_{0}$ of $G(\lambda)$ then it is also of the transfer function $\wh{G}(\lambda)$ and
\begin{equation}\label{dim1}
\text{dim}\;\mathcal{N}_{r}(G(\lambda_{0}))=\text{dim}\;\mathcal{N}_{r}(\wh{G}(\lambda_{0})).
\end{equation}
By \cite[Theorem 6.11]{strong}, $\det (\lambda_{0}I_{n}-A)\neq 0.$ Thus, by Proposition \ref{righteigen}, 
\begin{equation}\label{dim2}
\text{dim}\;\mathcal{N}_{r}(\wh{G}(\lambda_{0}))=\text{dim}\;\mathcal{N}_{r}(\mathcal{L}(\lambda_{0})) .
\end{equation}
By \eqref{unimodulareq} and Proposition \ref{lefteigen}, we have the same equalities for the dimensions of the left null-spaces, i.e., 
\begin{equation}\label{dimleft}
\text{dim}\;\mathcal{N}_{\ell}(G(\lambda_{0}))=\text{dim}\;\mathcal{N}_{\ell}(\wh{G}(\lambda_{0}))\quad \text{and}\quad \text{dim}\;\mathcal{N}_{\ell}(\wh{G}(\lambda_{0}))=\text{dim}\;\mathcal{N}_{\ell}(\mathcal{L}(\lambda_{0})) .
\end{equation}
Moreover, notice that since $G(\la)$ is square, $\dim\mathcal{N}_{r}(G(\lambda_{0}))=\dim\mathcal{N}_{\ell}(G(\lambda_{0})).$ 
	\end{rem}
	
	A consequence of Lemma \ref{lemmaright} is that we can recover very easily right eigenvectors of a rational matrix $G(\lambda)$ from the eigenvectors of the transfer function $\wh{G}(\lambda)$ of any $\M_{1}$-strong linearization of $G(\la).$ We state that in Theorem \ref{righteigenvector}, and we emphasize that this result is in the spirit of the one presented in \cite[Proposition 3.1]{ortho} for polynomial matrices $P(\la)$ and their strong linearizations in $\M_{1}(P).$
	
	\begin{theo}\label{righteigenvector} Let $G(\lambda)\in\FF(\la)^{m\times m}$ be a rational matrix with polynomial part of degree $k\geq 2,$ and let $\wh{G}(\lambda)$ be the transfer function of the $\M_{1}$-strong linearization $$\mathcal{L}(\lambda)= \left[
		\begin{array}{c|c}
		
		X(\lambda I_{n}-A)Y& 0_{n\times (k-1)m}\quad XB\\
		\hline \phantom{\Big|}
		
		-(v\otimes I_{m})CY& L(\lambda)
		\end{array}
		\right]$$
		 of $G(\lambda).$ Let $\lambda_{0}$ be a finite eigenvalue of $G(\lambda).$ Then, $u\in \mathcal{N}_{r}(G(\lambda_{0}))$ if and only if $\Phi_{k}(\lambda_{0})\otimes u \in \mathcal{N}_{r}(\wh{G}(\lambda_{0})) .$ Moreover, $\{ u_{1},\dots, u_{t}\} $ is a basis of $\mathcal{N}_{r}(G(\lambda_{0}))$ if and only if $\{\Phi_{k}(\lambda_{0})\otimes u_{1},\dots,\Phi_{k}(\lambda_{0})\otimes u_{t}\}$ is a basis of $\mathcal{N}_{r}(\wh{G}(\lambda_{0})).$ 
		
	\end{theo}
	\begin{proof} By Lemma \ref{lemmaright},
	$\wh{G}(\lambda_0)(\Phi_{k}(\lambda_0)\otimes I_{m})=v\otimes G(\lambda_0). $ Thus, it is easy to see that $u\in \mathcal{N}_{r}(G(\lambda_{0}))$ if and only if $\Phi_{k}(\lambda_{0})\otimes u \in \mathcal{N}_{r}(\wh{G}(\lambda_{0})) .$ Consider $\{u_{1},\ldots,u_{t}\}$ a basis of $\mathcal{N}_{r}(G(\lambda_{0})).$ Therefore, as $\text{dim}\;\mathcal{N}_{r}(G(\lambda_{0}))=\text{dim}\;\mathcal{N}_{r}(\wh{G}(\lambda_{0})),$ an immediate linear independence argument proves that $\{\Phi_{k}(\lambda_{0})\otimes u_{1},\dots,\Phi_{k}(\lambda_{0})\otimes u_{t}\}$ is a basis of $\mathcal{N}_{r}(\wh{G}(\lambda_{0})),$  and conversely.  
	\end{proof}
	
	In addition, by using Proposition \ref{righteigen}, we can recover the right eigenvectors of the transfer function $\wh{G}(\lambda)$ from the right eigenvectors of the linearization $\mathcal{L}(\lambda),$ and conversely. In particular, if $\left(\lambda_{0},\begin{bmatrix}y_0\\x_0\end{bmatrix}	\right)$ is a solution of the polynomial eigenvalue problem $\mathcal{L}(\lambda)z=0$ such that $\text{det}(\lambda_{0}I_{n}-A)\neq 0,$ then $(\lambda_{0},x_{0})$ is a solution of the rational eigenvalue problem $\wh{G}(\lambda)x=0.$ 
	
	In what follows, if we have a vector $\begin{bmatrix}y\\x\end{bmatrix},$ with $y\in\overline{\F}^{n\times 1}$ and $x\in\overline{\F}^{km\times 1},$  we will consider the vector $x$ partitioned as  $x=\begin{bmatrix}
	x^{(1)} &
	x^{(2)} &
	\cdots &
	x^{(k)}
	\end{bmatrix}^{T}$ with $ x^{(j)}\in \overline{\F}^{m\times 1}$ for $ j=1,\dots,k.$ Recall also in Theorem \ref{recoveryright1} that, as we have explained in Remark \ref{dimensiones}, if $\la_0\in\overline{\F}$ is a finite eigenvalue of $G(\la)$ then $\det(\la_0I_n-A)\neq 0.$ However, if $\la_0$ is an eigenvalue of $\mathcal{L}(\lambda),$ then, according to \cite[Theorem 6.11]{strong}, $\la_0$ might be a zero of $G(\la)$ that is simultaneously a pole and, therefore, $\det(\la_0 I_n-A) = 0,$ and $\la_0$ is not an eigenvalue of $G(\la).$ This is the reason why the condition $\det(\la_0I_n-A)\neq 0$ is assumed in parts $a)$ and $b)$ of Theorem \ref{recoveryright1}.    

	\begin{theo}\label{recoveryright1}\textbf{(Recovery of right eigenvectors from $\M_{1}$-strong linearizations)} Let $G(\lambda)\in\FF(\la)^{m\times m}$ be a rational matrix with polynomial part of degree $k\geq 2,$ and let  $$\mathcal{L}(\lambda)= \left[
		\begin{array}{c|c}
		
		X(\lambda I_{n}-A)Y& 0_{n\times (k-1)m}\quad XB\\
		\hline \phantom{\Big|}
		
		-(v\otimes I_{m})CY& L(\lambda)
		\end{array}
		\right]$$
		be an $\M_{1}$-strong linearization of $G(\lambda).$ 
		\begin{itemize}
			\item[a)] 	If $\left(\lambda_{0},\begin{bmatrix}y_0\\x_0\end{bmatrix}	\right)$ is a solution of the LEP $\mathcal{L}(\lambda)z=0$ such that $\det(\lambda_{0}I_{n}-A)\neq 0,$ then $(\lambda_{0},x_{0}^{(k)})$ is a solution of the REP $G(\lambda)x=0.$
			\item[b)] 	
			Moreover, if $\left\{\begin{bmatrix}y_1\\x_1\end{bmatrix},\ldots,\begin{bmatrix}y_t\\x_t\end{bmatrix}\right\}$ is a basis of $\mathcal{N}_{r}(\mathcal{L}(\lambda_{0})),$ with $\det(\lambda_{0}I_{n}-A)\neq 0,$ then $\{x_{1}^{(k)},\dots,x_{t}^{(k)}\}$ is a basis of $\mathcal{N}_{r}(G(\lambda_{0})).$ 
			\item[c)] 	Conversely, if $(\lambda_{0},u_{0})$ is a solution of the REP $G(\lambda)x=0,$ $x_{0}=\Phi_{k}(\lambda_{0})\otimes u_{0}$ and $y_{0}$ is defined as the unique solution of $(\lambda_{0}I_{n}-A)Yy_{0}+Bu_{0}=0,$ then $\left(\lambda_{0},\begin{bmatrix}y_0\\x_0\end{bmatrix}\right)$ is a solution of the LEP $\mathcal{L}(\lambda)z=0.$
			\item[d)]  Moreover, if $\{u_{1},\dots,u_{t}\}$ is a basis of $\mathcal{N}_{r}(G(\lambda_{0}))$ and, for $i=1,\ldots,t,$ $x_{i}=\Phi_{k}(\lambda_{0})\otimes u_{i}$ and $y_{i}$ is defined as the unique solution of $(\lambda_{0}I_{n}-A)Yy_{i}+Bu_{i}=0,$ then $\left\{\begin{bmatrix}y_1\\x_1\end{bmatrix},\ldots,\begin{bmatrix}y_t\\x_t\end{bmatrix}\right\}$ is a basis of $\mathcal{N}_{r}(\mathcal{L}(\lambda_{0})).$
		\end{itemize}

	\end{theo}
	
	\begin{proof} By Proposition \ref{righteigen}, if $\left(\lambda_{0},\begin{bmatrix}y_0\\x_0\end{bmatrix}	\right)$ is a solution of the LEP $\mathcal{L}(\lambda)z=0$ such that $\det(\lambda_{0}I_{n}-A)\neq 0,$ then $(\lambda_{0},x_{0})$ is a solution of the REP $\wh{G}(\lambda)x=0,$ where $\wh{G}(\la)$ is the transfer function matrix of $ \mathcal{L}(\lambda).$ By Theorem \ref{righteigenvector}, $x_{0}$ has the form $x_{0}=\Phi_{k}(\lambda_{0})\otimes u$ for some $u\in \mathcal{N}_{r}(G(\lambda_{0})).$ Since $\phi_{0}(\lambda)=1$ we have that $u=x_{0}^{(k)},$ which proves $a).$ The converse $c)$ is proved analogously.
	The implications $b)$ and $d)$ are consequences of $\eqref{dim1},$ $\eqref{dim2},$ basic arguments of linear independence, and the fact that $\mathcal{L}(\lambda_0)\begin{bmatrix}y_0\\x_0\end{bmatrix}=0$ if and only if $(\la_0 I_n - A)Yy_0 + XBx_0^{(k)}=0$ and $\wh{G}(\la_0)x_0=0.$
\end{proof}

Next, we pay attention to the recovery of left eigenvectors.

\begin{theo}\label{lefteigen1} \textbf{(Recovery of left eigenvectors from $\M_{1}$-strong linearizations)}\\ Let $G(\lambda)\in\FF(\la)^{m\times m}$ be a rational matrix with polynomial part of degree $k\geq 2,$ let $$\mathcal{L}(\lambda)= \left[
	\begin{array}{c|c}
	
	X(\lambda I_{n}-A)Y& 0_{n\times (k-1)m}\quad XB\\
	\hline \phantom{\Big|} 
	
	-(v\otimes I_{m})CY& L(\lambda)
	\end{array}
	\right]$$
be an $\M_{1}$-strong linearization of $G(\lambda),$ and let $\wh{G}(\lambda)$ be the transfer function of $\mathcal{L}(\lambda).$   
\begin{itemize}
	\item[a)] If $\left(\lambda_{0},\begin{bmatrix}y_0\\x_0\end{bmatrix}\right)$ is a solution of the LEP $z^{T}\mathcal{L}(\lambda)=0$ such that $\det(\lambda_{0}I_{n}-A)\neq 0,$ then $(\lambda_{0},(v^{T}\otimes I_{m})x_{0})$ is a solution of the REP $x^{T}G(\lambda)=0.$
	\item[b)] Moreover, if $\left\{ \begin{bmatrix}y_1\\x_1\end{bmatrix},\dots,\begin{bmatrix}y_t\\x_t\end{bmatrix} \right\}$ is a basis of $\mathcal{N}_{\ell}(\mathcal{L}(\lambda_{0})),$ with $\det(\lambda_{0}I_{n}-A)\neq 0,$ then $\{(v^{T}\otimes I_{m})x_{1},\dots,(v^{T}\otimes I_{m})x_{t}\}$ is a basis of $\mathcal{N}_{\ell}(G(\lambda_{0})).$
	\item[c)] Conversely, if $(\lambda_{0},u_{0})$ is a solution of the REP $x^{T}G(\lambda)=0,$ then there exists $x_{0}\in\mathcal{N}_{\ell}(\wh{G}(\lambda_{0}))$ such that $u_{0}=(v^{T}\otimes I_{m})x_{0}$ and if $y_{0}$ is defined as the unique solution of $y_{0}^{T}X(\lambda_{0}I_{n}-A)-u_{0}^{T}C=0,$ then $\left(\lambda_{0},\begin{bmatrix}y_0\\x_0\end{bmatrix}\right)$ is a solution of the LEP $z^{T}\mathcal{L}(\lambda)=0.$
	\item[d)] Moreover, if $\{u_{1},\dots,u_{t}\}$ is a basis of $\mathcal{N}_{\ell}(G(\lambda_{0}))$ then, for $i=1,\ldots,t,$ there exists $x_{i}\in\mathcal{N}_{\ell}(\wh{G}(\lambda_{0}))$ such that $u_{i}=(v^{T}\otimes I_{m})x_{i}$ and if $y_{i}$ is defined as the unique solution of $y_{i}^{T}X(\lambda_{0}I_{n}-A)-u_{i}^{T}C=0,$ then $\left\{\begin{bmatrix}y_1\\x_1\end{bmatrix},\dots,\begin{bmatrix}y_t\\x_t\end{bmatrix}\right\}$ is a basis of $\mathcal{N}_{\ell}(\mathcal{L}(\lambda_{0})).$ 
\end{itemize}

\end{theo}

\begin{proof}
	We consider the transfer function of $\mathcal{L}(\lambda),$ $\wh{G}(\lambda)=L(\lambda)+[0_{km\times (k-1)m}\quad (v\otimes I_{m})C(\lambda I_{n}-A)^{-1}B].$ If $\left(\lambda_{0},\begin{bmatrix}y_0\\x_0\end{bmatrix}\right)$ is a solution of the LEP $z^{T}\mathcal{L}(\lambda)=0$ such that $\det(\lambda_{0}I_{n}-A)\neq 0,$ by using Proposition \ref{lefteigen} $a)$ applied to $\mathcal{L}(\lambda),$ we get
	\begin{equation}\label{equ}
	x_{0}^{T}\wh{G}(\lambda_{0})=x_{0}^{T}L(\lambda_{0})+ [0_{1\times (k-1)m}\quad x_{0}^{T}(v\otimes I_{m})C(\lambda_{0} I_{n}-A)^{-1}B]=0,
	\end{equation}
	where $x_0\neq 0$ since $(\la_0,x_0)$ is a solution of the REP $x^{T}\wh{G}(\lambda)=0$ \footnote{With the notation of Proposition \ref{lefteigen}, it is easy to see that $[y_0^{T}\; x_0^T ]P(\la_0)=0$ if and only if $y_0^T A(\la_0)-x_0^T C(\la_0)=0$ and $x_0^{T}G(\la_0)=0.$ Thus, $x_0=0$ and $\det A(\la_0)\neq 0$ imply $y_0=0.$ Therefore, any left eigenvector of $P(\la)$ corresponding to the finite eigenvalue $\la_0$ must have $x_0\neq 0.$}.
	In addition, by Lemma \ref{lemmaright}, 
		$
		x_{0}^{T}\wh{G}(\lambda_{0})(\Phi_{k}(\lambda_{0})\otimes I_{m})=x_{0}^{T}(v\otimes I_{m}) G(\lambda_{0}).
		$
	Therefore $x_{0}^{T}(v\otimes I_{m}) G(\lambda_{0})=0.$ To see that $(v^{T}\otimes I_{m})x_{0}$ is a left eigenvector of $G(\lambda_{0}),$ we only need to prove that $x_{0}^{T}(v\otimes I_{m})\neq 0.$ Let us suppose that $x_{0}^{T}(v\otimes I_{m})=0,$ and let us get a contradiction. In this case $x_{0}^{T}(v\otimes I_{m})C(\lambda_{0} I_{n}-A)^{-1}B=0$ and, therefore, $x_{0}^{T}L(\la_{0})=x_{0}^{T}[v\otimes I_{m} \quad H]F_{\Phi}^{D}(\la_{0})=0$
	  by \eqref{equ}. We call $w^{T}=x_{0}^{T}[v\otimes I_{m} \quad H]$ and we consider $w$ partitioned as $w=(w_{i})_{i=1}^{k}$ with $w_i\in\overline{\efe}^{m\times 1}.$ We have that $w_{1}^{T}=x_{0}^{T}(v\otimes I_{m})=0.$ Therefore $[0\quad w_{2}^{T}\quad \cdots \quad w_{k}^{T}]F_{\Phi}^{D}(\la_{0})=0.$ This implies $-\alpha_{k-2}w_{2}^{T}=0$ and thus $w_{2}=0,$ since $\alpha_{k-2}\neq 0.$ Therefore $[0\quad 0\quad w_{3}^{T}\quad  \cdots \quad w_{k}^{T}]F_{\Phi}^{D}(\la_{0})=0$ and $w_{3}=0.$ Proceeding in this way it is easy to prove that $w_{i}=0$ for $i=2,\cdots,k.$ Thus $x_{0}^{T}[v\otimes I_{m} \quad H]=0$ which is a contradiction because $[v\otimes I_{m} \quad H]$ is assumed to be regular and $x_0\neq 0.$  This proves $a).$ 
	  
	  The implication $b)$ is proved as follows. From part $a),$ the vectors $(v^{T}\otimes I_{m})x_{1},\dots,(v^{T}\otimes I_{m})x_{t}$ belong to $\mathcal{N}_{\ell}(G(\la_{0})).$ Therefore, as a consequence of \eqref{dimleft}, if we prove that $\{(v^{T}\otimes I_{m})x_{1},\dots,(v^{T}\otimes I_{m})x_{t}\}$ is linearly independent, then $b)$ is proved. For this purpose, let $\alpha_1,\ldots,\alpha_t\in\overline{\F}$ be arbitrary scalars such that at least one is different from zero. Thus $0\neq \begin{bmatrix}\alpha_1y_1+\cdots +\alpha_t y_t \\\alpha_1x_1+\cdots +\alpha_t x_t\end{bmatrix}\in \mathcal{N}_{\ell}(\mathcal{L}(\la_{0})), $ and, from part $a),$ $x=(v^{T}\otimes I_m)(\alpha_1x_1+\cdots +\alpha_t x_t)\neq 0$ and $x\in \mathcal{N}_{\ell}(G(\la_{0})).$ 
	  
	  For proving $c),$ we prove first that there exists a basis of $\mathcal{N}_{\ell}(G(\la_{0}))$ of the form $\{(v^{T}\otimes I_{m})x_{1},\dots,(v^{T}\otimes I_{m})x_{t}\},$ where $\{x_{1},\dots,x_{t}\}$ is a basis of $\mathcal{N}_{\ell}(\wh{G}(\la_{0})).$ To this purpose, let $\left\{\begin{bmatrix}y_1\\x_1\end{bmatrix},\dots,\begin{bmatrix}y_t\\x_t\end{bmatrix}\right\}$ be a basis of $\mathcal{N}_{\ell}(\mathcal{L}(\la_{0})).$ Then, Proposition \ref{lefteigen} $b)$ applied to $\mathcal{L}(\la)$ implies that $\{x_{1},\dots,x_{t}\}$ is a basis of $\mathcal{N}_{\ell}(\wh{G}(\la_{0}))$ and Theorem \ref{lefteigen1} $b)$ that $\{(v^{T}\otimes I_{m})x_{1},\dots,(v^{T}\otimes I_{m})x_{t}\}$ is a basis of $\mathcal{N}_{\ell}(G(\la_{0})).$ Then, if $(\la_0,u_0)$ is a solution of the REP $x^{T}G(\la)=0,$ $u_0$ can be written as $u_0=(v^{T}\otimes I_m) \displaystyle\sum_{i=1}^{t}a_{i}x_{i}$ with $a_i\in\overline{\F},$ and we define $x_0= \displaystyle\sum_{i=1}^{t}a_{i}x_{i}\in\mathcal{N}_{\ell}(\wh{G}(\la_{0})).$ Finally, Proposition \ref{lefteigen} $c)$ applied to the solution $(\la_0,x_0)$ of the REP $x^{T}\wh{G}(\la)=0$ and to $\mathcal{L}(\la),$ and the fact that $\det(\la_0I_n-A)\neq 0$ imply that if $y_0$ is the unique solution of $y_{0}^{T}X(\lambda_{0}I_{n}-A)-x_{0}^{T}(v\otimes I_m)C=0,$ which is equivalent to $y_{0}^{T}X(\lambda_{0}I_{n}-A)-u_{0}^{T}C=0,$ then $\left(\lambda_{0},\begin{bmatrix}y_0\\x_0\end{bmatrix}\right)$ is a solution of the LEP $z^{T}\mathcal{L}(\lambda)=0.$ 
	  
	  Finally, the proof of $d)$ proceeds as follows. From part $c),$ we obtain that the vectors $x_1,\ldots , x_t$ satisfying $u_i=(v^T\otimes I_m)x_i$ exist, and that the vectors $\begin{bmatrix}y_1\\x_1\end{bmatrix},\dots,\begin{bmatrix}y_t\\x_t\end{bmatrix}$ belong to $\mathcal{N}_{\ell}(\mathcal{L}(\la_{0})).$ Therefore, taking into account \eqref{dimleft}, it only remains to prove that $\begin{bmatrix}y_1\\x_1\end{bmatrix},\dots,\begin{bmatrix}y_t\\x_t\end{bmatrix}$ are linearly independent. This is easily proved by contradiction:\\ If $\left\{\begin{bmatrix}y_1\\x_1\end{bmatrix},\dots,\begin{bmatrix}y_t\\x_t\end{bmatrix}\right\}$ is linearly dependent, then $\{x_{1},\dots,x_{t}\}$ is linearly dependent, and $\{u_{1},\dots,u_{t}\}$ is linearly dependent, which is a contradiction since $\{u_{1},\dots,u_{t}\}$ is a basis.
\end{proof}

\begin{rem}\label{singularcase2} \rm Analogously to Remark \ref{singularcase1}, if $G(\la)\in\F(\la)^{m\times m}$ is singular, then the results on null-spaces proved so far in Section \ref{eigenvectorfromm1} are valid for any $\la_0\in\overline{\F}$ that satisfies $\det (\la_0 I_n - A)\neq 0.$
\end{rem}

Finally, we study the recovery of the eigenvectors corresponding to the infinite eigenvalue from $\M_1$-strong linearizations.

\begin{theo}\label{infright} \textbf{(Recovery of eigenvectors associated to infinity from $\M_{1}$-strong linearizations)} Let $G(\lambda)\in\FF(\la)^{m\times m}$ be a rational matrix with polynomial part of degree $k\geq 2,$ let $$\mathcal{L}(\lambda)= \left[
	\begin{array}{c|c}
	
	X(\lambda I_{n}-A)Y& 0_{n\times (k-1)m}\quad XB\\
	\hline \phantom{\Big|}
	
	-(v\otimes I_{m})CY& L(\lambda)
	\end{array}
	\right]$$
	be an $\M_{1}$-strong linearization of $G(\lambda),$ and let $D_{k}$ be the leading matrix coefficient of the polynomial part of $G(\la)$ as in \eqref{polypart_polybasis}. Then the following statements hold:  
	\begin{itemize}
		\item[a)] $	\mathcal{N}_r (\rev G(0))=	\mathcal{N}_r (D_{k})$ and $x_{0}\in 	\mathcal{N}_r (D_{k})$ if and only if $\begin{bmatrix}0\\e_{1}\otimes x_{0}\end{bmatrix} \in\mathcal{N}_r (\rev \mathcal{L}(0)).$ Moreover,  $\{x_1,\ldots,x_q\}$ is a basis of $\mathcal{N}_r (\rev G(0))$ if and only if $\left\{\begin{bmatrix}0\\e_{1}\otimes x_{1}\end{bmatrix},\ldots,\right.$ $\left.\begin{bmatrix}0\\e_{1}\otimes x_{q}\end{bmatrix}\right\}$ is a basis of $\mathcal{N}_r (\rev \mathcal{L}(0)).$ 
		
		\item[b)] $	\mathcal{N}_{\ell} (\rev G(0))=	\mathcal{N}_{\ell} (D_{k})$ and $\begin{bmatrix}0\\ x_{0}\end{bmatrix} \in\mathcal{N}_{\ell} (\rev \mathcal{L}(0))$ if and only if $(v^{T}\otimes I_{m})x_{0}\in 	\mathcal{N}_{\ell} (D_{k}).$ Moreover, $\left\{\begin{bmatrix}0\\ x_{1}\end{bmatrix},\:\ldots\:,\begin{bmatrix}0\\ x_{q}\end{bmatrix}\right\}$ is a basis of $\mathcal{N}_{\ell} (\rev \mathcal{L}(0))$ if and only if  $\left\{(v^{T}\otimes I_{m})x_{1},\ldots,(v^{T}\otimes I_{m})x_{q}\right\}$ is a basis of $\mathcal{N}_{\ell} (\rev G(0)).$
	\end{itemize}
		
\end{theo}

\begin{proof} Notice that from \eqref{efe}, $$F_{\Phi}^{D}(\lambda)= \lambda\left[\begin{array}{cc}
	\alpha_{k-1}^{-1}D_{k}&0 \\
	0 & I_{(k-1)m}
	\end{array}\right] + F_{\Phi}^{D}(0).$$ We consider $$L(\la)=[v\otimes I_{m} \quad H]F_{\Phi}^{D}(\la)=[\alpha_{k-1}^{-1}(v\otimes D_{k})\quad H]\la+L(0)=:L_{1}\la + L_{0}$$ and let $\wh{G}(\la)$ be the transfer function matrix of $\mathcal{L}(\la).$ We have that $\rev \mathcal{L}(0)= \left[
	\begin{array}{c|c}
	
	XY& 0\\
	\hline 
	
	0 & L_{1}
	\end{array}
	\right]$ and $\rev \wh{G}(0)=\rev L(0)=L_{1}.$ Moreover, $\rev G(0)=\alpha_{0}^{-1}\alpha_{1}^{-1}\cdots \alpha_{k-1}^{-1}D_{k},$ that is, the coefficient of $\la^{k}$ in $D(\la).$ Therefore, $\mathcal{N}_r (\rev G(0))=	\mathcal{N}_r (D_{k}),$ $\mathcal{N}_\ell (\rev G(0))=	\mathcal{N}_\ell (D_{k})$ and $\infty$ is an eigenvalue of $G(\la)$ if and only if $D_{k}$ is singular. In addition, every right (respectively left) eigenvector $w$ of $\rev \mathcal{L}(0)$ has the form $w=\begin{bmatrix}0\\ x_{0}\end{bmatrix}$ for some $x_{0}\in 	\mathcal{N}_r(L_{1})$ (respectively $x_{0}\in 	\mathcal{N}_\ell(L_{1})$). By Lemma \ref{lemmaright}, we have
	\begin{equation*}
		\lambda \wh{G}\left(\dfrac{1}{\lambda}\right)\left(\la^{k-1}\Phi_{k}\left(\dfrac{1}{\lambda}\right)\otimes I_{m}\right)=v\otimes \la^{k} G\left(\dfrac{1}{\lambda}\right).
	\end{equation*}
Therefore,
\begin{equation*}
	\rev \wh{G}(0)(\rev \Phi_{k}(0)\otimes I_{m})=(v\otimes I_{m}) \rev G(0).
\end{equation*} 
Since $\rev \Phi_{k}(0)= \alpha_{0}^{-1}\alpha_{1}^{-1}\cdots \alpha_{k-2}^{-1} e_{1} ,$ we obtain 
	
\begin{equation*}\label{relinf}
	\alpha_{0}^{-1}\alpha_{1}^{-1}\cdots \alpha_{k-2}^{-1}\rev \wh{G}(0)(e_{1}\otimes I_{m})=(v\otimes I_{m}) \rev G(0).
\end{equation*} 
In addition, by \eqref{eq.unimodularoverla2}, there exist unimodular matrices $W_1(\la)$ and $W_2(\la)$ such that
\begin{equation*}\label{unimodinf}
W_1(0)\diag\left(\rev G(0),I_{(k-1)m}\right)
W_2(0)=\rev \wh{G}(0),
\end{equation*}
which implies that $\dim\mathcal{N}_r (\rev G(0))=\dim\mathcal{N}_r (\rev \wh{G}(0))$ and $\dim\mathcal{N}_{\ell} (\rev G(0))=\dim\allowbreak\mathcal{N}_{\ell} (\rev \wh{G}(0)).$
Finally $a)$ and $b)$ follow from the results above by using similar arguments to the ones we used in the recovery of eigenvectors associated to finite eigenvalues.
\end{proof}

\subsection{Eigenvectors from $\M_{2}$-strong linearizations}
If we proceed analogously as we did with $\M_{1}$-strong linearizations, and we use Lemma \ref{lemmaleft}, then we get Theorems \ref{recoverym2_1}, \ref{recoverym2_2} and \ref{recoverym2_3} to recover right and left eigenvectors of a rational matrix from those of its $\M_{2}$-strong linearizations. The proofs are essentially the same as those in Section \ref{eigenvectorfromm1} by interchanging the roles of left and right eigenvectors, and they are omitted for brevity.

\begin{lem}\label{lemmaleft} Let $G(\lambda)\in\FF(\la)^{m\times m}$ be a rational matrix with polynomial part of degree $k\geq 2,$ let $$\mathcal{L}(\lambda)= \left[
	\begin{array}{c|c}
	
	X(\lambda I_{n}-A)Y& XB(w^{T}\otimes I_{m})\\
	\hline \phantom{\Big|}
	
	\begin{array}{c}
	0_{(k-1)m\times n}\\
	-CY
	\end{array}& L(\lambda)
	\end{array}
	\right]$$
	be an $\M_{2}$-strong linearization of  $G(\lambda),$ and let $\wh{G}(\lambda)$ be the transfer function of $\mathcal{L}(\lambda).$ Then 
	\begin{equation}\label{transfer2}
	(\Phi_{k}(\lambda)^{T}\otimes I_{m})\wh{G}(\lambda)=w^{T}\otimes G(\lambda).
	\end{equation}
\end{lem}

\begin{theo}\label{recoverym2_1} \textbf{(Recovery of right eigenvectors from $\M_{2}$-strong linearizations)} Let $G(\lambda)\in\FF(\la)^{m\times m}$ be a rational matrix with polynomial part of degree $k\geq 2,$ let  $$\mathcal{L}(\lambda)= \left[
	\begin{array}{c|c}
	
	X(\lambda I_{n}-A)Y& XB(w^{T}\otimes I_{m})\\
	\hline \phantom{\Big|}
	
		\begin{array}{c}
		0_{(k-1)m\times n}\\
		-CY
		\end{array}& L(\lambda)
	\end{array}
	\right]$$
be an $\M_{2}$-strong linearization of $G(\lambda),$ and let $\wh{G}(\lambda)$ be the transfer function of $\mathcal{L}(\lambda).$  
\begin{itemize}
	\item[a)] If $\left(\lambda_{0},\begin{bmatrix}y_0\\x_0\end{bmatrix}	\right)$ is a solution of the LEP $\mathcal{L}(\lambda)z=0$ such that $\det(\lambda_{0}I_{n}-A)\neq 0 $ then, $(\lambda_{0},(w^{T}\otimes I_{m})x_{0})$ is a solution of the REP $G(\lambda)x=0.$ 
	\item[b)] Moreover, if $\left\{\begin{bmatrix}y_1\\x_1\end{bmatrix},\ldots,\begin{bmatrix}y_t\\x_t\end{bmatrix}\right\}$ is a basis of $\mathcal{N}_{r}(\mathcal{L}(\lambda_{0})),$ with $\det(\lambda_{0}I_{n}-A)\neq 0, $ then $\{(w^{T}\otimes I_{m})x_{1},\dots,(w^{T}\otimes I_{m})x_{t}\}$ is a basis of $\mathcal{N}_{r}(G(\lambda_{0})).$
	
	\item[c)] Conversely, if $(\lambda_{0},u_{0})$ is a solution of the REP $G(\lambda)x=0,$ then there exists $x_{0}\in\mathcal{N}_{r}(\wh{G}(\lambda_{0}))$ such that $u_{0}=(w^{T}\otimes I_{m})x_{0}$ and if $y_{0}$ is defined as the unique solution of $(\lambda_{0}I_{n}-A)Yy_{0}+Bu_{0}=0,$ then $\left(\lambda_{0},\begin{bmatrix}y_0\\x_0\end{bmatrix}\right)$ is a solution of the LEP $\mathcal{L}(\lambda)z=0.$  
	\item[d)] Moreover, if $\{u_{1},\dots,u_{t}\}$ is a basis of $\mathcal{N}_{r}(G(\lambda_{0}))$ then, for $i=1,\ldots,t,$ there exists $x_{i}\in\mathcal{N}_{r}(\wh{G}(\lambda_{0}))$ such that $u_{i}=(w^{T}\otimes I_{m})x_{i}$ and if $y_{i}$ is defined as the unique solution of $(\lambda_{0}I_{n}-A)Yy_{i}+Bu_{i}=0,$ then $\left\{\begin{bmatrix}y_1\\x_1\end{bmatrix},\dots,\begin{bmatrix}y_t\\x_t\end{bmatrix}\right\}$ is a basis of $\mathcal{N}_{r}(\mathcal{L}(\lambda_{0})).$
	
\end{itemize}
	
\end{theo}

\begin{theo}\label{recoverym2_2} \textbf{(Recovery of left eigenvectors from $\M_{2}$-strong linearizations)} Let $G(\lambda)\in\FF(\la)^{m\times m}$ be a rational matrix with polynomial part of degree $k\geq 2,$ and let  $$\mathcal{L}(\lambda)= \left[
	\begin{array}{c|c}
	
	X(\lambda I_{n}-A)Y& XB(w^{T}\otimes I_{m})\\
	\hline \phantom{\Big|}
	
	\begin{array}{c}
	0_{(k-1)m\times n}\\
	-CY
	\end{array}& L(\lambda)
	\end{array}
	\right]$$
	be an $\M_{2}$-strong linearization of $G(\lambda).$ 
	
	\begin{itemize}
		\item[a)] If $\left(\lambda_{0},\begin{bmatrix}y_0\\x_0\end{bmatrix}\right)$ is a solution of the LEP $z^{T}\mathcal{L}(\lambda)=0$ such that $\det(\lambda_{0}I_{n}-A)\neq 0,$ then $(\lambda_{0},x_{0}^{(k)})$ is a solution of the REP $x^{T}G(\lambda)=0.$

		\item[b)]  Moreover, if $\left\{\begin{bmatrix}y_1\\x_1\end{bmatrix},\dots,\begin{bmatrix}y_t\\x_t\end{bmatrix}\right\}$ is a basis of $\mathcal{N}_{\ell}(\mathcal{L}(\lambda_{0})),$ with $\det(\lambda_{0}I_{n}-A)\neq 0,$ then $\{x_{1}^{(k)},\dots,x_{t}^{(k)}\}$ is a basis of $\mathcal{N}_{\ell}(G(\lambda_{0})).$
		\item[c)] Conversely, if $(\lambda_{0},u_{0})$ is a solution of the REP $x^{T}G(\lambda)=0,$ $x_{0}=\Phi_{k}(\lambda_{0})\otimes u_{0}$ and $y_{0}$ is defined as the unique solution of $y_{0}^{T}X(\lambda_{0}I_{n}-A)-u_{0}^{T}C=0,$ then $\left(\lambda_{0},\begin{bmatrix}y_0\\x_0\end{bmatrix}\right)$ is a solution of the LEP $z^{T}\mathcal{L}(\lambda)=0.$
		\item[d)] Moreover, if $\{u_{1},\dots,u_{t}\}$ is a basis of $\mathcal{N}_{\ell}(G(\lambda_{0}))$ and, for $i=1,\ldots,t,$ $x_{i}=\Phi_{k}(\lambda_{0})\otimes u_{i}$ and $y_{i}$ is defined as the unique solution of $y_{i}^{T}X(\lambda_{0}I_{n}-A)-u_{i}^{T}C=0,$ then $\left\{\begin{bmatrix}y_1\\x_1\end{bmatrix},\ldots,\begin{bmatrix}y_t\\x_t\end{bmatrix}\right\}$ is a basis of $\mathcal{N}_{\ell}(\mathcal{L}(\lambda_{0})).$
	\end{itemize}
\end{theo}

\begin{rem} \rm Analogously to Remarks \ref{singularcase1} and \ref{singularcase2}, if $G(\la)\in\F(\la)^{m\times m}$ is singular, then the results on null-spaces in Theorems \ref{recoverym2_1} and \ref{recoverym2_2} hold for any $\la_0\in\overline{\F}$ such that $\det(\lambda_{0}I_{n}-A)\neq 0.$ 
\end{rem}

\begin{theo}\label{recoverym2_3} \textbf{(Recovery of eigenvectors associated to infinity from $\M_{2}$-strong linearizations)} Let $G(\lambda)\in\FF(\la)^{m\times m}$ be a rational matrix with polynomial part of degree $k\geq 2,$ let $$\mathcal{L}(\lambda)= \left[
	\begin{array}{c|c}
	
	X(\lambda I_{n}-A)Y& XB(w^{T}\otimes I_{m})\\
	\hline \phantom{\Big|}
	
	\begin{array}{c}
	0_{(k-1)m\times n}\\
	-CY
	\end{array}& L(\lambda)
	\end{array}
	\right]$$
	be an $\M_{2}$-strong linearization of $G(\lambda),$ and let $D_{k}$ be the leading matrix coefficient of the polynomial part of $G(\la)$ as in \eqref{polypart_polybasis}. Then the following statements hold:   
	\begin{itemize}
		\item[a)]  $	\mathcal{N}_r (\rev G(0))=	\mathcal{N}_r (D_{k})$ and $\begin{bmatrix}0\\ x_{0}\end{bmatrix} \in\mathcal{N}_r (\rev \mathcal{L}(0))$ if and only if $(w ^{T}\otimes I_{m})x_{0}\in 	\mathcal{N}_r (D_{k}).$ Moreover,   $\left\{\begin{bmatrix}0\\ x_{1}\end{bmatrix},\:\ldots\:,\begin{bmatrix}0\\ x_{q}\end{bmatrix}\right\}$ is a basis of $\mathcal{N}_r (\rev \mathcal{L}(0))$ if and only if  $\left\{(w^{T}\otimes I_{m})x_{1},\ldots,(w^{T}\otimes I_{m})x_{q}\right\}$ is a basis of $\mathcal{N}_r (\rev G(0))$.
		
		\item[b)] $	\mathcal{N}_{\ell} (\rev G(0))=	\mathcal{N}_{\ell} (D_{k})$ and $x_{0}\in 	\mathcal{N}_{\ell} (D_{k})$ if and only if $\begin{bmatrix}0\\e_{1}\otimes x_{0}\end{bmatrix} \in\mathcal{N}_{\ell} (\rev \mathcal{L}(0)).$ Moreover,  $\{x_1,\ldots,x_q\}$ is a basis of $\mathcal{N}_{\ell} (\rev G(0))$ if and only if $\left\{\begin{bmatrix}0\\e_{1}\otimes x_{1}\end{bmatrix},\ldots,\right.$ $\left.\begin{bmatrix}0\\e_{1}\otimes x_{q}\end{bmatrix}\right\}$ is a basis of $\mathcal{N}_{\ell} (\rev \mathcal{L}(0)).$

	\end{itemize}
	
\end{theo}

\section{Symmetric realizations of symmetric rational matrices}\label{sect:realsym}
In this section and in the next one our aim is to obtain a strong linearization of a \textit{symmetric rational matrix} $G(\lambda)\in \F(\lambda)^{m\times m},$ i.e., $G(\lambda)^{T}=G(\lambda),$ that preserves its symmetric structure. We write $G(\lambda)$ as
\begin{equation}\label{decomp}
G(\la)=D(\la)+G_{sp}(\la)
\end{equation}
 with $D(\la)$ its polynomial part and $G_{sp}(\la)$ its strictly proper part. Since \eqref{decomp} is a unique decomposition we obtain the following result just by taking transposes.
 \begin{prop}
Let $G(\lambda)\in \F(\lambda)^{m\times m}$ be a symmetric rational matrix. Then the matrices $D(\lambda)$ and $G_{sp}(\lambda)$ in \eqref{decomp} are also symmetric.
 \end{prop}

  Proposition \ref{symmetricsp} is the main result in this section and shows that any symmetric strictly proper rational matrix admits a state-space realization that reveals transparently the symmetry. In order to state concisely Proposition \ref{symmetricsp}, we will use the following definition.
 \begin{deff}\label{symmetricrel}
 	Let $G_{sp}(\lambda)\in \F(\lambda)^{m\times m}$ be a symmetric strictly proper rational matrix and let $n=\nu(G_{sp}(\lambda))$ be the least order of $G_{sp}(\la).$ A symmetric minimal state-space realization of $G_{sp}(\lambda)$ is an expression of the form
 	$$G_{sp}(\lambda)=W(S_{1}\lambda-S_{2})^{-1}W^{T}$$
 	where $S_{1},S_{2}\in\F^{n\times n}$ are symmetric matrices with $S_{1}$ nonsingular and $W\in\F^{m\times n}.$
 \end{deff}
  We remark that the realization described in Definition \ref{symmetricrel} is equivalent to \cite[Definition 4.4]{Antoulas} for a minimal state-space realization. However, in Definition \ref{symmetricrel} we express strictly proper matrices in a form more convenient for the goals of this paper. In particular, we will see in Section \ref{sect:sym} that by combining a symmetric minimal state-space realization of the matrix $G_{sp}(\lambda)$ in \eqref{decomp} and a symmetric strong block minimal bases pencil associated to $D(\lambda),$ we can construct symmetric strong linearizations of $G(\lambda).$ The next technical lemma is used in the proof of Proposition \ref{symmetricsp}.
\begin{lem}\label{ese}
	Let $G_{sp}(\lambda)\in \F(\lambda)^{m\times m}$ be a symmetric strictly proper rational matrix and let $G_{sp}(\lambda)=C(\lambda I_{n}-A)^{-1}B$ be a minimal state-space realization of $G_{sp}(\lambda).$ Then there exists a unique nonsingular and symmetric matrix $S\in \F^{n\times n}$ such that $A^{T}=S^{-1}AS \text{ and } C^{T}=S^{-1}B.$
\end{lem}

\begin{proof}
	As $G_{sp}(\lambda)$ is symmetric, $G_{sp}(\lambda)=B^{T}(\lambda I_{n}-A^{T})^{-1}C^{T}$ is also a minimal state-space realization of $G_{sp}(\lambda)$ since both have the same minimal order $n.$ Therefore, by \cite[Proposition 3.3.2]{real}, the realizations $(A,B,C)$ and $(A^{T},C^{T},B^{T})$ are similar and there exists a unique nonsingular matrix $S\in \F^{n\times n}$ such that
	\begin{equation}\label{similar}
	A^{T}=S^{-1}AS,\quad C^{T}=S^{-1}B,\quad B^{T}=CS.
	\end{equation}
	The fact that $(A,B,C)$ is a minimal realization of $G_{sp}(\lambda)$ is equivalent to that $(A,B)$ and $(A,C)$ are controllable and observable, respectively (see \cite[Chapter 3]{Rosen70}). That means that the controllability matrix of $(A,B)$ and the observability matrix of $(A,C),$ i.e., $$\mathcal{C}(
	A,B)=[B\quad AB\quad A^{2}B\quad\cdots\quad A^{n-1}B]\quad \text{and}\quad\mathcal{O}(A,C)=\left[\begin{array}{c}
	C\\
	CA\\
	CA^{2}\\
	\vdots\\
	CA^{n-1}
	\end{array}\right],$$ have both rank $n.$ From the equalities in \eqref{similar} it is easy to see that
	$S^{-1}\mathcal{C}(
	A,B)=\mathcal{O}(A,C)^{T},$ and $S^{-T}\mathcal{C}(
	A,B)=\mathcal{O}(A,C)^{T}.$
	As  $\mathcal{C}(
	A,B)$ has full row rank, we deduce that $S=S^{T}.$ 
\end{proof}	

	\begin{rem}\rm Notice that the system similarity matrix $S$ between the realizations in Lemma \ref{ese} is given by $S = \mathcal{O}(A,C)^{+}\mathcal{C}(A,B)^{T} = \mathcal{C}(A,B)( \mathcal{O}(A,C)^{T})^{\dagger}$
where $+$ denotes any left inverse and $\dagger$ denotes any right inverse. Notice also that these left and right inverses exist because $(A,B,C)$ is a minimal realization of $G_{sp}(\lambda)$ and that they can be taken to be the Moore--Penrose inverse. Thus $S$ can be efficiently computed when $\F=\R,\C.$
	\end{rem}
 
 \begin{prop}\label{symmetricsp} Any symmetric strictly proper rational matrix has a symmetric minimal state-space realization. 
 \end{prop}
 
\begin{proof} As said in Section \ref{prelim}, any strictly proper rational matrix $G_{sp}(\la)$ admits a minimal state-space realization, that is, $G_{sp}(\la)=C(\la I_{n} -A)^{-1}B$ \cite{Rosen70}. By Lemma \ref{ese}, there exists a unique nonsingular and symmetric matrix $S$ such that
$G_{sp}(\lambda)=C(\lambda I_{n}-A)^{-1}SC^{T}=C(\lambda S^{-1}-S^{-1}A)^{-1}C^{T},$ and $S^{-1}A$ is symmetric. 
\end{proof}

\begin{rem}\label{hankel} \rm We can construct a symmetric minimal state-space realization of a symmetric strictly proper rational matrix $G_{sp}(\lambda)\in\F(\la)^{m\times m}$ without previously considering a non-symmetric minimal state-space realization of $G_{sp}(\lambda),$ in contrast to what we have done in the proof of Proposition \ref{symmetricsp}. For this purpose we require $\efe$ not to be a field of characteristic $2.$
Let $G_{sp}(\lambda)=G_{1}\lambda^{-1}+G_{2}\lambda^{-2}+\cdots $ be the Laurent series of $G_{sp}(\lambda),$ which converges for $|\lambda|$ large enough. Let $n=\nu(G_{sp}(\la))$ be the least order of $G_{sp}(\la).$ We consider the block Hankel matrix
\begin{equation}\label{hankelmatrix}
H_{n}=\left[\begin{array}{cccc}
G_{1} & G_{2} & \cdots & G_{n} \\
G_{2} & G_{3} & \cdots & G_{n+1} \\
\vdots & \vdots & \ddots & \vdots \\
G_{n} & G_{n+1} & \cdots & G_{2n-1} 
\end{array}\right]
\end{equation}
and follow in a symmetric way the three steps of the algorithm in \cite[Section 3.4]{real} to get a symmetric minimal state-space realization from the Hankel matrix. Notice that the Hankel matrix is symmetric since $G_{sp}(\lambda)$ is symmetric, which implies $G_i=G_i^T$ for all $i\geq 1,$ and $\rank(H_{n})=n$ by \cite[Proposition 3.3.2]{real}. Therefore we can write
$$H_{n}=X\left[\begin{array}{cc}
K & 0 \\
0 & 0
\end{array}\right]X^{T}=X\left[\begin{array}{c}
K\\
0
\end{array}\right][I_{n} \quad 0]X^{T}$$
with $X$ nonsingular and $K\in\efe^{n\times n}$ diagonal (see \cite[Theorem 34.1]{duffee}). Let us denote $$\Gamma =X\left[\begin{array}{c}
K\\
0
\end{array}\right]\text{ and }\Lambda=[I_{n} \quad 0]X^{T}.$$ We have that $H_{n}=\Gamma\Lambda.$ We write 
$X= \left[\begin{array}{cc}
X_{1} & 
X_{2} 
\end{array}\right],$ where $X_{1}=  [X_{i1}]_{i=1}^{n}$
with $X_{i1}\in \F^{m\times n}$ for $i=1,\dots,n.$ Thus $$\Gamma=\left[ \begin{array}{c}
X_{11}K\\
\vdots\\
X_{n1}K
\end{array}\right]\text{ and }\Lambda=[X_{11}^{T}\quad \cdots \quad X_{n1}^{T}].$$ We define $$R=\left[\begin{array}{cccc}
G_{2} & G_{3} & \cdots & G_{n+1} \\
G_{3} & G_{4} & \cdots & G_{n+2} \\
\vdots & \vdots & \ddots & \vdots \\
G_{n+1} & G_{n+2} & \cdots & G_{2n} 
\end{array}\right] $$ and we set $ C=X_{11}K,$ $B=X_{11}^{T}$ and $A=\Gamma^{+} R \Lambda^{+},$ with $\Gamma^{+}=[K^{-1} \quad 0]X^{-1}$ and $\Lambda^{+}=X^{-T}\left[ \begin{array}{c}
I_{n}\\
0
\end{array}\right].$ Thus $A=[K^{-1} \quad 0]X^{-1}RX^{-T}\left[ \begin{array}{c}
I_{n}\\
0
\end{array}\right]$ and, by \cite[Theorem 3.4.1]{real}, $(A,B,C)$ is a minimal realization for $G_{sp}(\lambda).$ Therefore
\begin{equation*}
\begin{split}
G_{sp}(\lambda) & =X_{11}K\left(\lambda I_{n} - [K^{-1} \quad 0]X^{-1}RX^{-T}\left[ \begin{array}{c}
I_{n}\\
0
\end{array}\right]\right)^{-1}X_{11}^{T} \\ & = X_{11}\left(\lambda K^{-1} - [K^{-1} \quad 0]X^{-1}RX^{-T}\left[ \begin{array}{c}
K^{-1}\\
0
\end{array}\right]\right)^{-1}X_{11}^{T}.
\end{split}
\end{equation*}
Finally we set $W=X_{11},$ $S_{1}=K^{-1}$ and $S_{2}=[K^{-1} \quad 0]X^{-1}RX^{-T}\left[ \begin{array}{c}
K^{-1}\\
0
\end{array}\right],$ and we obtain a symmetric minimal state-space realization of $G_{sp}(\la).$

  In the particular, but very important in applications, case in which $G_{sp}(\lambda)\in\R(\lambda)^{m\times m},$ the Hankel matrix $H_n$ is a symmetric real matrix and, therefore, we can write
	$$H_{n}=P\left[\begin{array}{cc}
	K & 0 \\
	0 & 0
	\end{array}\right]P^{T}$$
	with $P$ \textit{orthogonal}, i.e., $P^{-1}=P^{T},$ and $K$ a diagonal matrix that has the eigenvalues of $H_{n}$ at the diagonal elements. In this case, let $P= \left[\begin{array}{cc}
	P_{1} & P_{2} 
	\end{array}\right],$ where $P_{1}=  [P_{i1}]_{i=1}^{n}$
	with $P_{i1}\in \F^{m\times n}$ for $i=1,\dots,n.$ Then we obtain $G_{sp}(\lambda)=P_{11}(\lambda K^{-1} - K^{-1}P_{1}^{T} R P_{1}K^{-1})^{-1}P_{11}^{T}.$
	That is, $G_{sp}(\la)$ has a symmetric minimal state-space realization $G_{sp}(\lambda)=W(\lambda S_{1}-S_{2})^{-1}W^{T}$ where $W=P_{11},$ $S_{1}=K^{-1}$ and $S_{2}=K^{-1}P_{1}^{T} R P_{1}K^{-1}.$ 
\end{rem}

From Proposition \ref{symmetricsp} and Remark \ref{hankel} we know how to write the strictly proper part $G_{sp}(\lambda)$ of a symmetric rational matrix $G(\lambda)$ as a symmetric minimal state-space realization with or without having in advance a particular non-symmetric minimal state-space realization of $G_{sp}(\lambda).$ Moreover, it is worth to emphasize that in many applications of symmetric REPs, this can be done very easily from the data of the model without any computational cost (see \cite[Section 8.3]{strong} or \cite[Section 4]{SuBai}).
 
\section{Symmetric strong linearizations for symmetric rational matrices }\label{sect:sym} 

In this section symmetric strong linearizations for symmetric rational matrices will be constructed. We start with Example \ref{symodddegree} in which we construct a symmetric strong linearization of a symmetric rational matrix when the polynomial part has odd degree. We will use Proposition \ref{symmetricsp} and a particular symmetric strong block minimal bases pencil associated to its polynomial part with sharp degree. After that, we present symmetric strong linearizations for symmetric rational matrices in which the polynomial part may have even or odd degree but the leading coefficient must be nonsingular. In order to get these results, we need to study symmetric strong linearizations in the polynomial case. 

\begin{example}\label{symodddegree}{\rm Let $G(\lambda)=D(\lambda)+G_{sp}(\lambda)\in\F(\lambda)^{m\times m}$ be a symmetric rational matrix. Consider the polynomial part $D(\lambda)$ written in terms of the monomial basis $D(\la)=D_k\la^k+D_{k-1}\la^{k-1}+\cdots+D_0\in\FF[\la]^{m\times m}$, with $k >1$ and $D_k \ne 0$, and the matrices 	\begin{equation}
		\label{eq:Lk}
		L_p(\lambda)=\begin{bmatrix}
		-1 & \lambda  \\
		& -1 & \lambda \\
		& & \ddots & \ddots \\
		& & & -1 & \lambda  \\
		\end{bmatrix}\in\mathbb{F}[\lambda]^{p\times(p+1)},
		\end{equation}
		and
		\begin{equation}
		\label{eq:Lambda}
		\Lambda_p(\lambda)^T =
		\begin{bmatrix}
		\lambda^{p} & \cdots & \lambda & 1
		\end{bmatrix} \in \FF[\lambda]^{1\times (p+1)}.
		\end{equation} A \textit{block Kronecker linearization} of $D(\la)$ is a pencil
		\begin{equation}
		\label{eq:blockKronPencil}
		\begin{array}{cl}
		L(\la) =
		\left[
		\begin{array}{c|c}
		M(\la) &L_{\eta}(\lambda)^{T}\otimes I_{m}\\\hline \phantom{\Big|}
		L_{\varepsilon}(\lambda)\otimes I_{m}&0
		\end{array}
		\right]&
		\begin{array}{l}
		\left. \vphantom{L_{\eta}^{T}(\lambda)\otimes I_{m}} \right\} {\scriptstyle (\eta+1)m}\\
		\left. \vphantom{L_{\varepsilon}(\lambda)\otimes I_{m}}\right\} {\scriptstyle \varepsilon m}
		\end{array}\\
		\hphantom{\mathcal{L}(\la)=}
		\begin{array}{cc}
		\underbrace{\hphantom{L_{\varepsilon}(\lambda)\otimes I_{n}}}_{(\varepsilon +1)m}&\underbrace{\hphantom{L_{\eta}^{T}(\lambda)\otimes I_{m}}}_{\eta m}
		\end{array}
		\end{array}
		\>
		\end{equation}
		such that $D(\la) = (\Lambda_{\eta}(\lambda)^T\otimes I_m)\, M(\lambda) \,(\Lambda_{\varepsilon}(\lambda)\otimes I_m)$ (see \cite[Definition 4.1]{BKL}). Recall that block Kronecker linearizations are particular cases of strong block minimal bases pencils \cite{BKL}. If the polynomial part $D(\lambda)$ has odd degree $k = 2 q+1$ we can consider the symmetric block Kronecker linearization in which 
		\[
		M(\la) = \begin{bmatrix} D_{2q +1} \la + D_{2q} & \\ & D_{2q -1} \la + D_{2q-2} \\ & & \ddots \\ & & & D_{1} \la + D_{0} \end{bmatrix}
		\] 
		 and $\varepsilon=\eta=q$. Proposition \ref{symmetricsp} allows us to write $G_{sp}(\la) =W(\lambda S_{1}-S_{2})^{-1}W^{T}$ with $S_{1}$ and $S_{2}$ symmetric and $S_{1}$ nonsingular. Applying \cite[Theorem 8.11]{strong} with $Y=-S_{1}X^{T}$ for any nonsingular matrix $X\in\F^{n\times n},$ $C=WS_1^{-1},$ $A=S_2S_1^{-1},$ $B=W^T,$ and $\widehat{K}_{1}=\widehat{K}_{2}=e_{q+1}^{T}\otimes I_{m},$ we obtain that the linear polynomial matrix
		\begin{equation*} 
		\mathcal{L}(\la) = \left[
		\begin{array}{c|c;{2pt/2pt}c}
		X(S_{2}-\lambda S_{1})X^{T} & \phantom{a}  0\quad XW^{T} \phantom{a} & 0 \\ \hline \phantom{\Big|}
		\begin{array}{c}
		0\\
		WX^{T}
		\end{array} \phantom{\Big|}& M(\la) & L_q(\la)^T\otimes I_{m} \\ \hdashline[2pt/2pt] \phantom{\Big|}
		0 & L_q (\la)\otimes I_{m} & 0
		\end{array}
		\right]
		\end{equation*}
		is a symmetric strong linearization of $G(\lambda).$
	}
\end{example}

\begin{rem}\rm The approach in Example \ref{symodddegree} can be extended to other symmetric strong block minimal bases pencils of the symmetric polynomial part $D(\la)$ of $G(\la)=G(\la)^T$ to construct other symmetric strong linearizations of $G(\la),$ as long as $D(\la)$ has odd-degree. See, for instance, the pencils considered in \cite{structured}. However, the linearization in Example \ref{symodddegree} is particularly simple and, in view of the results in \cite{odddegree}, we expect that it will have favourable numerical properties.  
\end{rem}

 Let $P(\lambda)\in\F[\lambda]^{m\times m}$ be a polynomial matrix of degree $k.$ A $km\times km$ pencil $L(\lambda)$ is called \textit{block-symmetric} if $L(\lambda)=L(\lambda)^{\mathcal{B}},$ where $L(\la)$ is viewed as a block partitioned pencil with $k\times k$ blocks each of them of size $m\times m.$ Notice that a pencil $L(\lambda)$ satisfies $L(\lambda)(\Phi_{k}(\lambda)\otimes I_{m})=v\otimes P(\lambda)$ for some vector $v\in\F^{k}$ if and only if $L(\lambda)^{\mathcal{B}}$ satisfies $(\Phi_{k}(\lambda)^{T}\otimes I_{m})L(\lambda)^{\mathcal{B}}=v^{T}\otimes P(\lambda).$ Thus, if $L(\lambda)\in\M_{1}(P)$ is block-symmetric, then $L(\lambda)\in\mathbb{M}_{1}(P)\cap\mathbb{M}_{2}(P).$ This intersection space was introduced in \cite{ortho}, it is called \textit{double generalized ansatz space}, and it is denoted by $$\mathbb{DM}(P)=\mathbb{M}_{1}(P)\cap\mathbb{M}_{2}(P).$$

If $\{\phi_{j}(\lambda)\}_{j=0}^{\infty}$ is the monomial basis, the space $\DM(P)$ is denoted $\mathbb{DL}(P)$ and was introduced originally in \cite{MMMM}. In \cite[Corollary 6]{ortho} it is shown that if a pencil $L(\lambda)$ belongs to $\DM(P)$ then its right and left ansatz vectors are the same, which is called simply \textit{ansatz vector}, and that $$\mathbb{DM}(P)=\{L(\lambda)\in\mathbb{M}_{1}(P):L(\lambda)=L(\lambda)^{\mathcal{B}}\}.$$

In fact, if $P(\lambda)\in\F[\lambda]^{m\times m}$ is a symmetric polynomial matrix we obtain that any pencil in $\DM(P)$ must be symmetric. This result is not in \cite{ortho}, and we state it in Theorem \ref{symmetric}. For its proof, we use Lemmas \ref{zero} and \ref{bijective}.

\begin{lem}\label{zero}
	Let $P(\lambda)\in\efe[\la]^{m\times m}$ be a polynomial matrix of degree $k\geq 2$ and let $L(\lambda)\in \mathbb{DM}(P)$ with ansatz vector $0\in \F^{k}.$ Then $L(\lambda)=0.$
\end{lem}

\begin{proof} Notice that from \eqref{efe},
	\begin{equation}\label{asterisco}
F_{\Phi}^{P}(\lambda)= \lambda\left[\begin{array}{cc}
\alpha_{k-1}^{-1}P_{k}&0 \\
0 & I_{(k-1)m}
\end{array}\right] + F_{\Phi}^{P}(0),
	\end{equation}
	  where $P(\la)$ is expressed as in \eqref{polynomial}. From \cite[Corollary 6]{ortho}, $L(\lambda)$ must have the form 
	\begin{equation*} 
	L(\lambda)=[0_{km\times m}\quad H]F_{\Phi}^{P}(\lambda)  =[0\quad H]\lambda + [0 \quad H]F_{\Phi}^{P}(0) =\left[\begin{array}{c}
				0\\
				H^{\mathcal{B}}
			\end{array}\right]\lambda + F_{\Phi}^{P}(0)^{\mathcal{B}}\left[\begin{array}{c}
			0\\
			H^{\mathcal{B}}
		\end{array}\right]. 
\end{equation*}
As $L(\lambda)$ is block symmetric, $H$ must have the form $H=\left[\begin{array}{c}
0\\
W
\end{array}\right]$ where $W$ is a $(k-1)m\times (k-1)m$ block symmetric matrix. Let $W=[W_{ij}]_{i,j=1}^{k-1}$ with $W_{ij}\in\F^{m\times m}.$ Then
$$[0_{km\times m}\quad H]F_{\Phi}^{P}(0)=\left[\begin{array}{cccc}
0 & 0 & \cdots & 0\\
-\alpha_{k-2}W_{11} & * & \cdots & *\\
-\alpha_{k-2}W_{21} & * & \cdots & *\\
\vdots\\
-\alpha_{k-2}W_{(k-1)1} & * & \cdots & *
\end{array}\right].$$
Notice that $[0_{km\times m}\quad H]F_{\Phi}^{P}(0)$ is also block symmetric because of the block symmetry of $L(\lambda).$ Then, we obtain that $W_{i1}=W_{1i}=0$ with $i=1,\dots , k-1.$ Next, we proceed by induction. Let $j\in\{2,\dots,k-1\}$ and suppose that $W_{it}=W_{ti}=0$ for all $i=1,\dots,k-1$ and $t=1,\dots,j-1.$ Then, 
$$[0_{km\times m}\quad H]F_{\Phi}^{P}(0)=\bbordermatrix{& & & (j-1) \cr
		& 0 &\cdots & 0 & 0 & 0 & \cdots & 0\cr
		&	\vdots & & \vdots & \vdots & \vdots & & \vdots \cr
	(j)	&	0 &\cdots & 0 & 0 & 0 & \cdots & 0\cr
	 &	0 &\cdots & 0 &   -\alpha_{k-(j+1)}W_{jj} & * & \cdots & *\cr
	 &	0 &\cdots & 0 &  -\alpha_{k-(j+1)}W_{(j+1)j} & * &  \cdots & *\cr
	&	\vdots & & \vdots & \vdots & \vdots & & \vdots \cr
	 &	0 &\cdots & 0 & 	-\alpha_{k-(j+1)}W_{(k-1)j} & * &  \cdots & *\cr}.$$
Therefore, $W_{ij}=W_{ji}=0$ with $i=1,\dots , k-1.$ By induction, $H=0$ and $L(\lambda)=0.$\end{proof}

 Theorem 3.4 in \cite{sympoly} states that for each $v\in\F^{k}$ there is a uniquely determined pencil in $\mathbb{DL}(P)$ with ansatz vector $v.$ We show this result extended to the space $\DM(P)$ in the following lemma. 

\begin{lem}\label{bijective}
	Let $P(\lambda)\in\F[\lambda]^{m\times m}$ be a polynomial matrix of degree $k\geq 2.$ For each $v\in\F^{k}$ there is only one pencil in $\DM(P)$ with ansatz vector $v.$
\end{lem}

\begin{proof}
	We consider the linear map $ \DM(P)\longrightarrow \F^{k}$ that associates to any pencil $L(\lambda)$ in $\DM(P)$ its ansatz vector $v\in\F^{k}.$ By Lemma \ref{zero} this map is injective and by \cite[Corollary 7]{ortho} $\text{dim}(\DM(P))=k.$ Therefore, the map is bijective.
\end{proof}

Let $P(\lambda)\in\F[\lambda]^{m\times m}$ be a symmetric polynomial matrix, and let us define the set
$$\bS(P)=\{L(\lambda)\in\M_{1}(P):L(\lambda)=L(\lambda)^{T}\}.$$
The elements in $\bS(P)$ are in $\DM(P)$ because if $L(\lambda)=[v\otimes I_{m}\quad H]F_{\Phi}^{P}(\lambda)\in \bS(P)$ then $L(\lambda)^{T}=F_{\Phi}^{P}(\lambda)^{\mathcal{B}}\left[\begin{array}{c}
v^{T}\otimes I_{m} \\
H^{T}
\end{array}\right]\in\mathbb{M}_{2}(P),$ since in the case $P(\lambda)$ is symmetric $F_{\Phi}^{P}(\lambda)^{T}=F_{\Phi}^{P}(\lambda)^{\mathcal{B}},$ and $L(\la)=L(\la)^{T}.$ Moreover, Theorem \ref{symmetric} shows that $\bS(P)$ and $\DM(P)$ are equal.

\begin{theo}\label{symmetric} Let $P(\lambda)\in\F[\lambda]^{m\times m}$ be a symmetric polynomial matrix of degree $k\geq 2.$ Then $$\mathbb{DM}(P)=\bS(P).$$
	
\end{theo}

\begin{proof} We have already seen that $\bS(P)\subseteq \DM(P).$ To see the other inclusion we only have to use Lemma \ref{bijective} and \cite[Corollary 6]{ortho}, and notice that if $L(\lambda)\in\DM(P)$ with $P(\la)$ symmetric then 

$$L(\lambda)=[v\otimes I_{m}\quad H]F_{\Phi}^{P}(\lambda)=F_{\Phi}^{P}(\lambda)^{\mathcal{B}}\left[\begin{array}{c}
v^{T}\otimes I_{m} \\
H^{\mathcal{B}}
\end{array}\right],$$
and
$$L(\lambda)^{T}=[v\otimes I_{m}\quad (H^{\mathcal{B}})^{T}]F_{\Phi}^{P}(\lambda)=F_{\Phi}^{P}(\lambda)^{\mathcal{B}}\left[\begin{array}{c}
v^{T}\otimes I_{m} \\
H^{T}
\end{array}\right],$$
which implies that $L(\la)^{T}\in\DM(P)$ and that $L(\la)$ and $L(\la)^{T}$ have the same ansatz vector. So, by Lemma \ref{bijective}, $L(\lambda)=L(\lambda)^{T}$ and $L(\lambda)\in\bS(P).$ 
\end{proof}
Therefore, if $P(\lambda)$ is a symmetric polynomial matrix all the pencils in $\DM(P)$ are also symmetric. In order to find linearizations in $\DM(P)$ we have to consider only regular polynomials $P(\lambda)$ because by \cite[Theorem 7]{ortho} if $P(\lambda)
$ is a singular polynomial matrix then none of the pencils in $\DM(P)$ is a linearization for $P(\lambda).$\\

In Theorem \ref{strongsymmetric}, we construct symmetric strong linearizations for a symmetric rational matrix from a particular symmetric strong linearization of its polynomial part $D(\lambda)$ when the leading coefficient $D_{k}$ of $D(\lambda)$ is nonsingular. This particular strong linearization is the pencil in $\mathbb{DM}(D)$ with ansatz vector $e_{k},$ i.e., the last vector in the canonical basis of $\efe^{k}.$ Some properties of this pencil are studied in Lemma \ref{leading}.

\begin{lem}\label{leading}
	Let $D(\lambda)\in\F[\lambda]^{m\times m}$ be a polynomial matrix with degree $k\geq 2.$ Let $L(\lambda)=[e_{k}\otimes I_{m}\quad H]F_{\Phi}^{D}(\la)\in \DM(D).$ Then $[e_{k}\otimes I_{m}\quad H]$ is nonsingular if and only if the leading matrix coefficient $D_{k}$ of $D(\lambda)$ is nonsingular. Moreover, if $D(\la)$ is regular, $L(\la)$ is a strong linearization of $D(\lambda)$ if and only if $D_{k}$ is nonsingular.
\end{lem}

\begin{proof}
	Let $L(\lambda)=[e_{k}\otimes I_{m}\quad H]F_{\Phi}^{D}(\la)\in\DM(D).$ We write, by using \eqref{asterisco} and \cite[Corollary 6]{ortho}, 	
	\begin{equation*}
	\begin{split}
	L(\lambda)=[e_{k}\otimes I_{m}\quad H]F_{\Phi}^{D}(\lambda) & =[e_{k}\otimes \alpha_{k-1}^{-1}D_{k}\quad H]\lambda + [e_{k}\otimes I_{m} \quad H]F_{\Phi}^{D}(0)\\
	& =\left[\begin{array}{c}
	e_{k}^{T}\otimes \alpha_{k-1}^{-1}D_{k}\\
	H^{\mathcal{B}}
	\end{array}\right]\lambda + F_{\Phi}^{D}(0)^{\mathcal{B}}\left[\begin{array}{c}
	e_{k}^{T}\otimes I_{m}\\
	H^{\mathcal{B}}
	\end{array}\right].
	\end{split}
	\end{equation*}
	Then $H=\left[\begin{array}{c}
	0_{m\times (k-2)m} \quad \alpha_{k-1}^{-1}D_{k}\\
	\begin{array}{c}
	H^{'}
	\end{array}
	\end{array}\right]$ for some $(k-1)m\times (k-1)m$ block symmetric matrix $H^{'}.$ Let $H^{'}=[H_{ij}^{'}]_{i,j=1}^{k-1}$ with $H_{ij}^{'}\in\F^{m\times m}.$ If we calculate the first block row and block column of the product $[e_{k}\otimes I_{m} \quad H]F_{\Phi}^{D}(0)$ we obtain 
	$$\left[\begin{array}{cccccc}
	0 & 0 & \cdots & 0 & -\frac{\alpha_{0}}{\alpha_{k-1}}D_{k} & -\frac{\beta_{0}}{\alpha_{k-1}}D_{k}\\
	-\alpha_{k-2}H_{11}^{'} & * & \cdots & * & *& *\\
	-\alpha_{k-2}H_{21}^{'} & * & \cdots & * & *& *\\
	\vdots\\
	-\alpha_{k-2}H_{(k-2)1}^{'} & * & \cdots & *& * & *\\
	-\frac{\beta_{k-1}}{\alpha_{k-1}}D_{k}+D_{k-1}-\alpha_{k-2}H_{(k-1)1}^{'} & * & \cdots & *& * & *
	\end{array}\right].$$
	Since $[e_{k}\otimes I_{m} \quad H]F_{\Phi}^{D}(0)$ is block symmetric we obtain
	
	\begin{equation}\label{zeroblock}
	H_{1i}^{'}=H_{i1}^{'}=0 \text{ for } i=1,\dots,k-3 
	\end{equation} 
	and $$-\alpha_{k-2}H_{(k-2)1}^{'}=-\frac{\alpha_{0}}{\alpha_{k-1}}D_{k}.$$
	Thus,
	\begin{equation}\label{h}
	H_{(k-2)1}^{'}=H_{1(k-2)}^{'}=\dfrac{\alpha_{0}}{\alpha_{k-1}\alpha_{k-2}}D_{k}.
	\end{equation}
	Using \eqref{zeroblock} and \eqref{h} and calculating the second block row and block column of the product $[e_{k}\otimes I_{m} \quad H]F_{\Phi}^{D}(0)$ as before, we obtain
	\begin{equation*}
	H_{2i}^{'}=H_{i2}^{'}=0 \text{ for } i=1,\dots,k-4 
	\end{equation*} 
	and $$-\alpha_{k-3}H_{(k-3)2}^{'}=-\alpha_{1}H_{1(k-2)}^{'}.$$
	Thus,
	\begin{equation*}
	H_{(k-3)2}^{'}=H_{2(k-3)}^{'}=\dfrac{\alpha_{0}\alpha_{1}}{\alpha_{k-1}\alpha_{k-2}\alpha_{k-3}}D_{k}.
	\end{equation*}
	In general, an induction argument proves that
	\begin{equation*}
	H_{(k-j)i}^{'}=H_{i(k-j)}^{'}=\dfrac{\alpha_{0}\alpha_{1}\cdots\alpha_{i-1}}{\alpha_{k-1}\alpha_{k-2}\cdots\alpha_{k-j}}D_{k}\text{ for }j-i=1,
	\end{equation*}
	and the matrix $[e_{k}\otimes I_{m}\quad H]$ has the following block anti-triangular form

	$$\left[\begin{array}{ccccccc}
	0 & 0  &\cdots & 0 & 0 & 0 & \alpha_{k-1}^{-1}D_{k} \\
	0 & 0 & \cdots & 0 & 0 &\dfrac{\alpha_{0}}{\alpha_{k-1}\alpha_{k-2}}D_{k} & *\\
	0 & 0 & \cdots  & 0 &\dfrac{\alpha_{0}\alpha_{1}}{\alpha_{k-1}\alpha_{k-2}\alpha_{k-3}}D_{k} & * &* \\
	\vdots & & &  \iddots & &  \\
	0 & 0 & \dfrac{\alpha_{0}\alpha_{1}}{\alpha_{k-1}\alpha_{k-2}\alpha_{k-3}}D_{k} & * & * & * & * \\
	0 & \dfrac{\alpha_{0}}{\alpha_{k-1}\alpha_{k-2}}D_{k} & * & * &* & * &*\\
	I_{m} & * & * &* & * & * & *
	\end{array}\right]
	. $$
	Therefore $[e_{k}\otimes I_{m}\quad H]$ is nonsingular if and only if $D_{k}$ is nonsingular. Moreover, if $D(\la)$ is regular, $[e_{k}\otimes I_{m}\quad H]$ is nonsingular if and only if $L(\la)$ is a strong linearization of $D(\lambda)$ by \cite[Theorem 3]{ortho}.
\end{proof}

Computing the pencil in $\DM(P)$ with ansatz vector $e_k,$ or with any other ansatz vector $v,$ may be difficult. In general, one can follow the procedure in \cite[Section 7]{ortho} or use the MATLAB code in \cite[Subsection 7.1]{bivariate}. However, if the recurrence relation \eqref{recu} is simple and $k$ is low, then the computation can be performed easily by hand, as we illustrate in Example \ref{chevyshev}. 

\begin{example}\label{chevyshev} \rm
	For a second degree polynomial matrix $D(\lambda)=D_{2}\phi_{2}(\lambda)+D_{1}\phi_{1}(\lambda)+D_{0}\phi_{0}(\lambda)$ expressed in terms of a polynomial basis satisfying \eqref{recu}, the pencil $L(\lambda)\in \mathbb{DM}(D)$ with ansatz vector $e_{2}$ is

	\begin{equation*}
	L(\lambda)=\left[\begin{array}{cc}
	-\frac{\alpha_{0}}{\alpha_{1}}D_{2} & \frac{\lambda-\beta_{0}}{\alpha_{1}}D_{2}\\
	\frac{\lambda-\beta_{0}}{\alpha_{1}}D_{2} & \left(\frac{\beta_{0}-\beta_{1}}{\alpha_{0}\alpha_{1}}(\la-\beta_0)-\frac{\gamma_{1}}{\alpha_1}\right)D_{2}+\frac{\lambda-\beta_{0}}{\alpha_{0}}D_{1}+D_{0}
	\end{array}\right].
	\end{equation*}
	This can be obtained, for instance, by computing the matrix $H^{'}$ as in the proof of Lemma \ref{leading}. For example, Chebyshev polynomials of the first kind $\{\phi_{j}(\lambda)\}_{j=0}^{\infty}$ satisfy the following three-term recurrence relation:
	
	\begin{equation}
	\frac{1}{2}\phi_{j+1}(\lambda)=\lambda\phi_{j}(\lambda)-\frac{1}{2}\phi_{j-1}(\lambda) \quad j\geq 1
	\end{equation}
	\\
	where $\phi_{-1}(\lambda)=0,$ $\phi_{0}(\lambda)=1$ and $\phi_{1}(\lambda)=\lambda.$ Therefore, $\alpha_{0}=1,$ $\alpha_{j}=\gamma_{j}=\frac{1}{2}$ for $j\geq 1,$ $\beta_{j}=0$  $j\geq 0$ and 
	\begin{equation*}
	L(\lambda)=\left[\begin{array}{cc}
	-2D_{2} & 2\lambda D_{2}\\
	2\lambda D_{2} & \lambda D_{1} + D_{0}-D_{2} 
	\end{array}\right].
	\end{equation*}
	Chebyshev polynomials of the second kind satisfy the same recurrence relation with $\phi_{1}(\lambda)=2\lambda.$ Thus, $\alpha_{j}=\gamma_{j}=\frac{1}{2},$ $\beta_{j}=0$ for $j\geq 0$ and
	\begin{equation*}
	L(\lambda)=\left[\begin{array}{cc}
	-D_{2} & 2\lambda D_{2}\\
	2\lambda D_{2} & 2\lambda D_{1} + D_{0}-D_{2}
	\end{array}\right].
	\end{equation*}
	For a cubic polynomial matrix $D(\lambda)=D_{3}\phi_{3}(\lambda)+D_{2}\phi_{2}(\lambda)+D_{1}\phi_{1}(\lambda)+D_{0}\phi_{0}(\lambda)$ expressed in terms of Chebyshev polynomials of the first kind, the pencil $L(\lambda)\in \mathbb{DM}(D)$ with ansatz vector $e_{3}$ is
	
	\begin{equation*}
	L(\lambda)=\left[\begin{array}{ccc}
	0 & -2D_{3} & 2\lambda D_{3}\\
	-2D_{3} & 4\lambda D_{3} -2 D_{2} & 2\lambda D_{2}-2 D_{3}\\
	2\lambda D_{3} & 2\lambda D_{2}-2D_{3} & \lambda(D_{1}+D_{3}) + D_{0}-D_{2}
	\end{array}\right]. 
	\end{equation*}
	If $D(\lambda)$ is expressed in terms of Chebyshev polynomials of the second kind we obtain 
	\begin{equation*}
	L(\lambda)=\left[\begin{array}{ccc}
	0 & - D_{3} & 2\lambda D_{3}\\
	-D_{3} & 2\lambda D_{3} - D_{2} & 2\lambda D_{2}- D_{3}\\
	2\lambda D_{3} & 2\lambda D_{2}-D_{3} & 2\lambda D_{1} + D_{0}-D_{2}
	\end{array}\right]. 
	\end{equation*}

\end{example}

By using Theorems \ref{muno} or \ref{mdos}, Theorem \ref{symmetric}, Lemma \ref{leading} and Proposition \ref{symmetricsp}, we obtain in Theorem \ref{strongsymmetric} symmetric strong linearizations of a symmetric rational matrix when the leading coefficient of its polynomial part is nonsingular as we announced.

\begin{theo}\label{strongsymmetric}
 Let $G(\lambda)\in \F(\lambda)^{m\times m}$ be a symmetric rational matrix and let $G(\lambda)=D(\lambda)+G_{sp}(\lambda)$ be its unique decomposition into its polynomial part $D(\lambda)\in\F[\lambda]^{m\times m}$ and its strictly proper part $G_{sp}(\lambda)\in\F(\lambda)^{m\times m}.$ Assume that $\deg(D(\la))=k\geq 2$ and let $n=\nu(G(\la)).$ Consider a symmetric minimal state-space realization of $G_{sp}(\lambda),$ i.e., $G_{sp}(\lambda)=W(\lambda S_{1}-S_{2})^{-1}W^{T}$ as in Definition \ref{symmetricrel}, and $L(\lambda)\in \DM(D)$ with ansatz vector $e_{k}.$ Let $\mu\in\F,$ $\mu\neq 0.$ If the leading matrix coefficient $D_{k}$ of $D(\lambda)$ is nonsingular then, for any nonsigular matrix $Z\in  \F^{n\times n},$ the linear polynomial matrix
 \begin{equation}\label{symmetriclin}
\mathcal{L}(\lambda)= \left[
\begin{array}{c|c}

\mu Z(S_{2}-\lambda S_{1})Z^{T}&  0_{n\times (k-1)m}\quad \mu ZW^{T}\\
\hline \phantom{\Big|}

\begin{array}{c}
0_{(k-1)m\times n}\\
\mu WZ^{T}
\end{array}& \mu L(\lambda)
\end{array}
\right]
 \end{equation}
is a symmetric strong linearization of $G(\lambda).$ 
\end{theo}

\begin{proof}Let $L(\lambda)=[e_{k}\otimes I_{m}\quad H]F_{\Phi}^D(\la)$ be the pencil in $\DM(D)$ with ansatz vector $e_{k}.$ Let us denote $L_{\mu}(\lambda)$ the pencil in $\mathbb{DM}(D)$ with ansatz vector $\mu e_{k},$ i.e., $L_{\mu}(\lambda)=\mu L(\lambda).$ Since $D_{k}$ is nonsingular, the matrix $[e_{k}\otimes I_{m}\quad H]$ is also nonsingular by using Lemma \ref{leading}. Notice that if $G_{sp}(\lambda)=W(\lambda S_{1}-S_{2})^{-1}W^{T}$ is a symmetric minimal state-space realization of $G_{sp}(\lambda)$ then $G_{sp}(\lambda)=W(\lambda I_{n}-S_{1}^{-1}S_{2})^{-1}S_{1}^{-1}W^{T}$ is a minimal state-space realization. It only remains to consider Theorem \ref{muno} with $X=\mu ZS_{1}$ and $Y=-Z^{T}.$  Equivalently, we can consider Theorem \ref{mdos} with $X= ZS_{1}$ and $Y=-\mu Z^{T}.$ 
\end{proof}

\begin{example} \rm
	
			Let $G(\lambda)\in \F(\lambda)^{m\times m}$ be a symmetric rational matrix and write $G(\lambda)=D(\lambda)+G_{sp}(\lambda)$ as sum of its polynomial part and its strictly proper part.
		 Suppose that
		$$
			D(\la)=D_{k}\la^{k}+D_{k-1}\la^{k-1}+ \cdots +D_{1}\la+ D_{0},$$
		with $k\geq 2$ and $D_{k}$ nonsingular, and write
		$G_{sp}(\lambda)=W(\lambda S_{1}-S_{2})^{-1}W^{T}$ as a symmetric minimal state-space realization.
			For the monomial basis we obtain by \cite[Theorem 3.5]{sympoly} that the pencil $L(\lambda)\in \mathbb{DL}(D)$ with ansatz vector $e_{k}$ is
			
			$$L(\lambda)= \lambda \left[ \begin{array}{ccccc}
			& & & & D_{k}\\
			& & & \iddots & D_{k-1}\\
			& & \iddots &\iddots & \vdots\\
			&\iddots &\iddots &  & D_{2}\\
			D_{k}  & D_{k-1} & \cdots & D_{2} & D_{1}
			
			\end{array}\right]-\left[\begin{array}{ccccc}
			& & & D_{k} & \\
			& &\iddots & D_{k-1} &\\
			& \iddots &\iddots & \vdots &\\
			D_{k}  & D_{k-1}  & \cdots & D_{2} &\\
			& & & &  -D_0
			\end{array}
			\right].$$
		Then, by Theorem \ref{strongsymmetric}, the linear polynomial matrix
		
		\begin{equation*}
		\mathcal{L}(\lambda)= \left[
		\begin{array}{c|c}
		
		S_{2}-\lambda S_{1}&  0_{n\times (k-1)m}\quad  W^{T}\\
		\hline 
		
		\begin{array}{c}
		0_{(k-1)m\times n}\\
		W
		\end{array}& L(\lambda)
		\end{array}
		\right]
		\end{equation*}
		is a symmetric strong linearization of $G(\lambda).$
\end{example}

We can obtain infinitely many symmetric strong linearizations by using Theorem \ref{strongsymmetric} and Lemma \ref{more}.

\begin{con}\label{newsymmetric} Under the same assumptions as in Theorem \ref{strongsymmetric}, consider the symmetric strong linearization $\mathcal{L}(\lambda)$ in \eqref{symmetriclin}. Let $Q\in\F^{n\times n},$ $P\in\F^{km\times km}$ be nonsingular matrices and $R\in\F^{km\times n}.$ Then 	$$\wh{\mathcal{L}}(\lambda)=\left[ \begin{array}{cc}
	Q & 0 \\
	R & P
	\end{array} \right]\mathcal{L}(\lambda)\left[ \begin{array}{cc}
	Q^{T} & R^{T} \\
	0 & P^{T}
	\end{array} \right]$$ is a symmetric strong linearization of $G(\lambda).$
\end{con}

\section{Hermitian strong linearizations for Hermitian rational matrices}\label{sect:herm}
In this section we extend the results in Sections \ref{sect:realsym} and \ref{sect:sym} from symmetric to Hermitian rational matrices. Since most of the arguments are similar to those in the symmetric case, we limit ourselves to state the main results, and most of the proofs are ommitted. We consider the ring of polynomials $\C[\la]$ and a polynomial basis $\{\phi_{j}(\lambda)\}_{j=0}^{\infty}$ that satisfies the three-term recurrence relation:
\begin{equation*}
\alpha_{j}\phi_{j+1}(\lambda)=(\lambda-\beta_{j})\phi_{j}(\lambda)-\gamma_{j}\phi_{j-1}(\lambda) \quad j\geq 0
\end{equation*}
as in \eqref{recu}, with $\alpha_{j},\beta_{j},\gamma_{j}\in\R,$ $\alpha_{j}\neq 0,$ $\phi_{-1}(\lambda)=0,$ and $\phi_{0}(\lambda)=1.$  Let $P(\lambda)\in\C[\lambda]^{m\times m}$ be a polynomial matrix of degree $k$ written in terms of this basis, i.e., $P(\lambda)=\displaystyle\sum_{i=0}^{k} P_{i}\phi_{i}(\lambda)$ with $P_{i}\in\C^{m\times m}.$ Suppose that $P(\la)$ is \textit{Hermitian}, i.e., $P(\la)^{*}=P(\overline{\la})$ or, equivalently, $P^{*}(\la)=P(\la),$ where $P^{*}(\la)$ is defined as $P^*(\lambda)=\displaystyle\sum_{i=0}^{k} P_{i}^*\phi_{i}(\lambda)$ with $P_i^*$ the conjugate transpose of $P_i\in\C^{m\times m}.$  We also consider the set of pencils 
$$\mathbb{H}(P)=\{ \la X+Y\in\mathbb{M}_{1}(P): X^{*}=X, Y^{*}=Y\}.$$ That is, $\mathbb{H}(P)$ is the set of pencils in $\M_{1}(P)$ that are Hermitian. Theorem \ref{hermpencils} shows that the elements of $\mathbb{H}(P)$ are in $\DM(P),$ and that, in fact, they are the pencils in $\DM(P)$ with real ansatz vector. The proof of Theorem \ref{hermpencils} is ommitted for brevity since it is similar to the proof of \cite[Theorem 6.1]{sympoly}, which is Theorem \ref{hermpencils} in the particular case $\phi_j(\la)=\la^j$ for $j\geq 0.$
\begin{theo}\label{hermpencils}
	Let $P(\la)\in\C[\la]^{m\times m}$ be a Hermitian polynomial matrix. Then $\mathbb{H}(P)$ is the subset of all pencils in $\DM(P)$ with real ansatz vector.
\end{theo}
Let $G(\la)\in\C(\la)^{m\times m}$ be a Hermitian rational matrix, i.e., a rational matrix satisfying $G(\la)^*=G(\overline{\la}).$ Consider $G(\la)=D(\la)+G_{sp}(\la)$ as in \eqref{eq.polspdec}. Then $D(\la)$ and $G_{sp}(\la)$ are also Hermitian. For Hermitian strictly proper rational matrices we introduce the notion of Hermitian minimal state-space realizations, in the spirit of Definition \ref{symmetricrel}. 

\begin{deff}\label{hermitianrel}
	Let $G_{sp}(\lambda)\in \C(\lambda)^{m\times m}$ be a Hermitian strictly proper rational matrix and let $n=\nu(G_{sp}(\lambda)).$ A Hermitian minimal state-space realization of $G_{sp}(\lambda)$ is an expression of the form
	$$G_{sp}(\lambda)=W(\lambda H_{1}-H_{2})^{-1}W^{*}$$
	where $H_{1},H_{2}\in\C^{n\times n}$ are Hermitian matrices, with $H_{1}$ nonsingular, and $W\in\C^{m\times n}.$
\end{deff}

Following arguments similar to those in Lemma \ref{ese} and Proposition \ref{symmetricsp}, it is easy to see that the strictly proper part of a Hermitian rational matrix has a Hermitian minimal state-space realization. 

\begin{prop}\label{hermsp} Any Hermitian strictly proper rational matrix has a Hermitian minimal state-space realization. 
\end{prop}
\begin{proof}
In order to obtain a Hermitian minimal state-space realization of $G_{sp}(\la),$ we can consider a minimal state-space realization $G_{sp}(\lambda)=C(\lambda I_{n}-A)^{-1}B.$ We prove analogously to Lemma \ref{ese} that there exists a unique nonsingular and Hermitian matrix $H\in \C^{n\times n}$ such that $A^{*}=H^{-1}AH \text{ and } C^{*}=H^{-1}B.$ Therefore, $G_{sp}(\lambda)=C(\lambda H^{-1}-H^{-1}A)^{-1}C^{*}$ is a Hermitian minimal state-space realization of $G_{sp}(\la).$
\end{proof}
\begin{rem}\rm 
	 Another constructive way to prove Proposition \ref{hermsp} is to consider the Hankel matrix $H_{n}$ of $G_{sp}(\la)$ defined in \eqref{hankelmatrix}, that is also Hermitian, and write 
	 $$H_{n}=U\left[\begin{array}{cc}
	 K & 0 \\
	 0 & 0
	 \end{array}\right]U^{*}$$
	 with $U$ \textit{unitary}, i.e., $U^{-1}=U^{*},$ and $K$ a diagonal matrix that has the eigenvalues of $H_{n}$ at the diagonal elements. Then proceed as in the last paragraph of Remark \ref{hankel} to get a Hermitian minimal state-space realization. Notice that $K$ is Hermitian because the eigenvalues of $H_{n}$ are real.
\end{rem}
 By using Proposition \ref{hermsp} and Theorem \ref{hermpencils}, we obtain in Theorem \ref{stronghermitian} Hermitian strong linearizations of a Hermitian rational matrix when the leading coefficient of its polynomial part is nonsingular, analogously as we did in Theorem \ref{strongsymmetric} for the symmetric case.
	
	\begin{theo}\label{stronghermitian}
		Let $G(\lambda)\in \C(\lambda)^{m\times m}$ be a Hermitian rational matrix and let $G(\lambda)=D(\lambda)+G_{sp}(\lambda)$ be its unique decomposition into its polynomial part $D(\lambda)\in\C[\lambda]^{m\times m}$ and its strictly proper part $G_{sp}(\lambda)\in\C(\lambda)^{m\times m}.$ Assume that $\deg(D(\la))=k\geq 2$ and let $n=\nu(G(\la)).$ Consider a Hermitian minimal state-space realization of $G_{sp}(\lambda),$ i.e., $G_{sp}(\lambda)=W(\lambda H_{1}-H_{2})^{-1}W^{*}$ as in Definition \ref{hermitianrel}, and $L(\lambda)\in \DM(D)$ with ansatz vector $e_{k}.$ Let $\mu\in\R,$ $\mu\neq 0.$ If the leading matrix coefficient $D_{k}$ of $D(\lambda)$ is nonsingular then, for any nonsigular matrix $Z\in  \C^{n\times n},$ the linear polynomial matrix
		\begin{equation}\label{eq_Her}
		\mathcal{L}(\lambda)= \left[
		\begin{array}{c|c}
		
		\mu Z(H_{2}-\lambda H_{1})Z^{*}&  0_{n\times (k-1)m}\quad \mu ZW^{*}\\
		\hline \phantom{\Big|}
		
		\begin{array}{c}
		0_{(k-1)m\times n}\\
		\mu WZ^{*}
		\end{array}& \mu L(\lambda)
		\end{array}
		\right]
		\end{equation}
	is a Hermitian strong linearization of $G(\lambda).$ 
	\end{theo}
As in Corollary \ref{newsymmetric}, we can obtain infinitely many Hermitian strong linearizations by using Theorem \ref{stronghermitian} and Lemma \ref{more}.

\begin{con}\label{newhermitian} Under the same assumptions as in Theorem \ref{stronghermitian}, consider the Hermitian strong linearization $\mathcal{L}(\lambda)$ in \eqref{eq_Her}. Let $Q\in\C^{n\times n},$ $P\in\C^{km\times km}$ be nonsingular matrices and $R\in\C^{km\times n}.$ Then 	$$\wh{\mathcal{L}}(\lambda)=\left[ \begin{array}{cc}
	Q & 0 \\
	R & P
	\end{array} \right]\mathcal{L}(\lambda)\left[ \begin{array}{cc}
	Q^{*} & R^{*} \\
	0 & P^{*}
	\end{array} \right]$$ is a Hermitian strong linearization of $G(\lambda).$
\end{con}

\section{Strong linearizations of rational matrices with polynomial part expressed in other polynomial bases}\label{sect:other}

Polynomial bases $\{\phi_{j}(\lambda)\}_{j=0}^{\infty}$ satisfying a three-term recurrence relation as in \eqref{recu} are by far the most useful in applications. However, from a theoretical point of view, a natural question is whether or not the results in this paper can be extended to other polynomial bases. The goal of this section is to show that this can be done by using exactly the same tools that we have used in previous sections, that is, \cite[Theorem 8.11]{strong}, our key Lemma \ref{more}, and the results in \cite{ortho}. Since the arguments in this section are very similar to the ones previously used, we will simply sketch the main ideas.

Let $D(\la)$ be the polynomial part of a rational matrix $G(\la)\in\efe(\la)^{m\times m},$ with $\deg(D(\la))\allowbreak=k\geq 2.$ Let us consider, motivated by \eqref{efe} and its properties, a polynomial basis $\{\psi_{j}(\lambda)\}_{j=0}^{\infty}$ of $\F[\lambda],$ with $\psi_{j}(\lambda)$ a polynomial of degree $j,$  that satisfies a linear relation:
\begin{equation}\label{linearrelation}
M_{\Psi}(\la) \Psi_{k}(\lambda) =0,
\end{equation}
where $M_{\Psi}(\la)\in\efe[\la]^{(k-1)\times k}$ is a minimal basis with all its row degrees equal to $1,$ and  $\Psi_{k}(\lambda)=[\psi_{k-1}(\lambda)\;\cdots\; \psi_{1}(\lambda)\text{ }\psi_{0}(\lambda)]^{T}$ with $\Psi_{k}(\lambda_0)\neq 0$ for all $\la_0\in\overline{\efe} .$ Then there exists a vector $w\in\efe^{k}$ such that 
\begin{equation}\label{unimodular}
U(\lambda)=\left[ {\begin{array}{cc}
	M_{\Psi}(\lambda)\otimes I_{m} \\
	w^T\otimes I_{m}
	\end{array} } \right]
\end{equation}
is unimodular, and its inverse has the form $U(\lambda)^{-1}=[\widehat{\Psi}_{k}(\lambda)\quad \Psi_{k}(\lambda)\otimes I_{m}]$
with $\widehat{\Psi}_{k}(\lambda)\in\F[\lambda]^{km\times (k-1)m}$ (see \cite[Lemma 8.4]{strong}). Let
\begin{equation}\label{efe_strongblock}
F_{\Psi}^{D}(\la)=\left[\begin{array}{c}
m_{\Psi}^{D}(\la)\\
M_{\Psi}(\la)\otimes I_{m}
\end{array}\right]\in \F[\lambda]^{km\times km}
\end{equation}
be a pencil such that $m_{\Psi}^{D}(\la)(\Psi_{k}(\lambda)\otimes I_{m})=D(\la).$ Then, $F_{\Psi}^{D}(\lambda)$ is a strong block minimal bases pencil associated to $D(\lambda)$ with sharp degree which verifies $F_{\Psi}^{D}(\lambda)(\Psi_{k}(\lambda) \otimes I_{m})=e_{1}\otimes D(\lambda).$ Thus, we can apply \cite[Theorem 8.11]{strong} and Lemma \ref{more} in order to construct strong linearizations of $G(\la)$ from pencils of the form $L(\la)=[v\otimes I_{m}\quad H]F_{\Psi}^{D}(\lambda)$ with $v\in\F^{k}$ and $[v\otimes I_{m}\quad H]$  nonsingular. Notice that pencils $L(\la)$ of this form verify the ansatz relation $L(\lambda)(\Psi_{k}(\lambda)\otimes I_{m})=v\otimes D(\lambda).$ In summary, with those arguments, we obtain the following result that is the generalization of Theorem \ref{muno} for polynomial bases as in \eqref{linearrelation}. 

\begin{theo}\label{otherbases} Let $G(\lambda)\in \F(\lambda)^{m\times m}$ be a rational matrix written as in \eqref{eq.polspdec}, and let $G_{sp}(\lambda)=C(\lambda I_{n}-A)^{-1}B$ be a minimal order state-space realization of $G_{sp}(\lambda).$ Assume that $deg(D(\lambda))\geq 2$ and write $D(\lambda)$ in terms of a polynomial basis $\{\psi_{j}(\lambda)\}_{j=0}^{\infty}$ satisfying \eqref{linearrelation}, as
	\begin{equation}\label{express_otherbases}
D(\lambda)=D_{k}\psi_{k}(\lambda)+D_{k-1}\psi_{k-1}(\lambda)+\cdots + D_{1}\psi_{1}(\lambda)+D_{0}\psi_{0}(\lambda)
	\end{equation}
	with $D_{k}\neq 0.$ Let $L(\la)=[v\otimes I_{m}\quad H]F_{\Psi}^{D}(\lambda)$ with $[v\otimes I_{m}\quad H]$ nonsingular and $F_{\Psi}^{D}(\la)$ as in \eqref{efe_strongblock}. Let $w\in\efe^{k}$ be the vector in \eqref{unimodular}. Then, for any nonsingular matrices $X,Y\in\F^{n\times n}$ the linear polynomial matrix
	$$\mathcal{L}(\lambda)= \left[
	\begin{array}{c|c}
	
	X(\lambda I_{n}-A)Y&  XB(w^T\otimes I_m)\\
	\hline \phantom{\Big|}
	
	\begin{array}{c}
	-(v\otimes I_m)CY\\
	\end{array}&L(\lambda)
	\end{array}
	\right]$$
	is a strong linearization of $G(\lambda).$
	
\end{theo}

In a similar manner, the results in Sections \ref{sect:m2} and \ref{recovery} can be extended to square rational matrices with polynomial parts expressed in terms of polynomial bases as in \eqref{linearrelation}.  

In Example \ref{poly_basis2} we consider degree-graded polynomial bases presented in \cite[Section 9]{ortho}, and we construct strong linearizations of square rational matrices by expressing the polynomial parts in terms of these bases and using Theorem \ref{otherbases}. 

\begin{example}\label{poly_basis2} \rm Let  $\{\psi_{j}(\lambda)\}_{j=0}^{\infty}$ be a degree-graded polynomial basis of $\F[\lambda]$ that satisfies the following recurrence relation:
 \begin{equation*}\label{recu2}
 \psi_{j}(\lambda)=(\lambda-\alpha_{j})\psi_{j-1}(\lambda)+\displaystyle\sum_{i=0}^{j-2}\beta_{j}^{i}\psi_{i}(\lambda) \quad j\geq 1
 \end{equation*}
 where $\alpha_{j}\in\efe$ for $j\geq 1,$ $\beta_{j}^{i}\in\efe$ for $j\geq 2,$ $0\leq i\leq j-2$ and $\psi_{0}(\lambda)=1.$ Let $G(\la)=D(\la) + C(\lambda I_{n}-A)^{-1}B$  be an $m\times m$ rational matrix written as in Theorem \ref{otherbases}. We express the polynomial part $D(\la)$ in terms of the polynomial basis $\{\psi_{j}(\lambda)\}_{j=0}^{\infty},$ as in \eqref{express_otherbases}. Let us denote $\Psi_{k}(\lambda)=[\psi_{k-1}(\lambda)\;\cdots\; \psi_{1}(\lambda)\text{ }\psi_{0}(\lambda)]^{T}$ and consider the following pencil $G_{\Psi}^{D}(\lambda)$ introduced in \cite[Section 9]{ortho}: $$G_{\Psi}^{D}(\lambda)=\left[ {\begin{array}{cc}
 	m_{\Psi}^{D}(\lambda) \\
 	M_{\Psi}(\lambda)\otimes I_{m} \\
 	\end{array} } \right]\in \F[\lambda]^{km\times km},$$ where
 \begin{equation*}
 m_{\Psi}^{D}(\lambda)=\left[ (\lambda -\alpha_{k})D_{k}+D_{k-1}\quad \beta_{k}^{k-2}D_{k}+D_{k-2}\quad  \cdots\quad \beta_{k}^{1}D_{k}+ D_{1} \quad \beta_{k}^{0}D_{k}+ D_{0}\right],
 \end{equation*}
 and
 \begin{equation*}
 M_{\Psi}(\lambda)=
 \left[ {\begin{array}{cccccccc}
 	-1 & (\lambda -\alpha_{k-1}) & \beta_{k-1}^{k-3} & \beta_{k-1}^{k-4} &\cdots & \beta_{k-1}^{2} & \beta_{k-1}^{1} & \beta_{k-1}^{0}   \\
 	& -1 & (\lambda - \alpha_{k-2}) & \beta_{k-2}^{k-4} &\cdots & \beta_{k-2}^{2} & \beta_{k-2}^{1} & \beta_{k-2}^{0} \\
 	& &\ddots &\ddots & \ddots  &  \vdots & \vdots & \vdots  \\
 	& & & & &-1& (\lambda-\alpha_{2}) & \beta_{2}^{0} \\
 	& & & & & &-1& (\lambda-\alpha_{1})
 	\end{array} } \right] .
 \end{equation*}
 The matrix $G_{\Psi}^{D}(\lambda)$ verifies that 
 $G_{\Psi}^{D}(\lambda)(\Psi_{k}(\lambda) \otimes I_{m})=e_{1}\otimes D(\lambda).$
 Moreover, $G_{\Psi}^{D}(\lambda)$ is a strong block minimal bases pencil associated to $D(\lambda)$ with sharp degree. It can be proved, as in \cite[Theorem 1]{ortho}, that any pencil $L(\la)$ that verifies $L(\lambda)(\Psi_{k}(\lambda)\otimes I_{m})=v\otimes D(\lambda)$ for some vector $v\in\F^{k}$ can be written as $L(\lambda)=[v\otimes I_{m}\quad H]G_{\Psi}^{D}(\lambda)$
 for some matrix $H\in\F^{km \times (k-1)m}.$ If we consider a pencil $L(\lambda)$ of this form with $[v\otimes I_{m}\quad H]$ nonsingular we can obtain strong linearizations for $G(\la).$ In particular, we have that conditions in Theorem \ref{otherbases} hold, and we can apply it with $w=e_k.$ Then, we have that for any nonsingular matrices $X,Y\in\F^{n\times n}$ the linear polynomial matrix
 $$\mathcal{L}(\lambda)= \left[
 \begin{array}{c|c}
 
 X(\lambda I_{n}-A)Y& 0_{n\times (k-1)m}\quad XB\\
 \hline \phantom{\Big|}
 
 \begin{array}{c}
 -(v\otimes I_m)CY\\
 \end{array}&L(\lambda)
 \end{array}
 \right]$$
 is a strong linearization of $G(\lambda).$ 
\end{example}

\section{Conclusions and future work}\label{sect:con}

As a consequence of the definitions and the theory developed in \cite{strong}, we have proved the simple Lemma \ref{more}, which allows us to construct infinitely many strong linearizations of any rational matrix $G(\la)\in\efe(\la)^{p\times m}$ from any given strong linearization of $G(\la).$ This result has been combined with some of the strong linearizations of a rational matrix $G(\la)\in\efe(\la)^{m\times m}$ constructed in \cite[Theorem 8.11]{strong} and with the strong linearizations of its polynomial part presented in \cite{ortho} to create new families of strong linearizations of square rational matrices. The recovery of the eigenvectors of the rational matrix from those of the linearizations in these families has been thoroughly investigated, as well as the preservation of symmetric and Hermitian structures of the rational matrix in the linearizations. 

We are convinced that the techniques developed in this paper together with the results in \cite{strong} can be applied to solve essentially all the following problems: How to construct a strong linearization of a rational matrix $G(\la)$ expressed as the sum of its polynomial part $D(\la)$ and its strictly proper part $G_{sp}(\la),$ given a strong linearization of $D(\la)$ in any of the families of strong linearizations of polynomial matrices developed in the last years and a minimal order state-space realization of $G_{sp}(\la).$ In particular, we hope that these techniques will allow to construct strong linearizations of rational matrices preserving structures that are different from the symmetric and Hermitian structures. However, we emphasize that, although any rational matrix can be expressed as the sum of its polynomial and strictly proper parts, this expression may not be easily available from the applications and/or may not be the best representation in a particular problem. Therefore, the development of strong linearizations of rational matrices starting from other representations is a problem that will be investigated in the future.

\end{document}